%% file: HdimV2.tex
\documentclass[english]{amsart}
\usepackage[T1]{fontenc}
\usepackage[utf8]{inputenc}
\usepackage{microtype}
\usepackage{amsmath}
\usepackage{amssymb}
\usepackage{amsthm}
\usepackage{mathrsfs}
\usepackage[all]{xy}
\usepackage{tikz}
\usepackage{subcaption}
\usepackage{graphicx}
\usetikzlibrary{automata,positioning,calc,intersections,through,backgrounds,patterns,fit,external}
\usepackage[autostyle,english=british]{csquotes}
\usepackage[backend=biber,style=alphabetic,doi=false,isbn=false,url=false]{biblatex}
	\renewbibmacro{in:}{\ifentrytype{article}{}{\printtext{\bibstring{in}\intitlepunct}}}
\usepackage[colorlinks]{hyperref}
	\hypersetup{hidelinks}		% links are in plain black
\usepackage{bookmark}
\usepackage[capitalise,noabbrev]{cleveref}

\theoremstyle{plain}
\newtheorem{theorem}{Theorem}
\newtheorem{lemma}[theorem]{Lemma}

\newtheorem{corollary}[theorem]{Corollary}
\newtheorem{conjecture}[theorem]{Conjecture}
\theoremstyle{definition}
\newtheorem{definition}[theorem]{Definition}
\theoremstyle{remark}
\newtheorem*{remark}{Remark}
\newtheorem{convention}[theorem]{Convention}

\newcommand{\numberset}{\mathbb}

\newcommand{\NN}{\numberset{N}}

\newcommand{\RR}{\numberset{R}}

\newcommand{\cL}{\mathcal{L}}

\renewcommand{\epsilon}{\varepsilon}

\newcommand{\alphabet}{\mathcal{A}}

\newcommand{\rauzy}{$\mathcal R$}
\newcommand{\zorich}{$\mathcal Z$}

\begin{document}

\title[Double rotations of infinite type]{A note on double rotations of infinite type}
\author[M. Artigiani]{Mauro Artigiani}
\address{School of Engineering, Science and Technology\\
 Universidad del Rosario\\
  Carrera 6 No.\ 12 C-16\\
	Bogot\'a 111711\\
	Colombia}
\email{mauro.artigiani@urosario.edu.co}

\author[C. Fougeron]{Charles Fougeron}
\address{IRIF, Universit\'e de Paris \\
	8 Place Nemours, 75013 \\
	Paris\\
	France}
\email{charles.fougeron@math.cnrs.fr}

\author[P. Hubert]{Pascal Hubert}
\address{Institut de Math\'ematiques de Marseille\\
	39 rue F. Joliot-Curie\\
	13453 Marseille Cedex 20\\
	France}
\email{pascal.hubert@univ-amu.fr}

\author[A. Skripchenko]{Alexandra Skripchenko}
\address{Faculty of Mathematics\\
 National Research University Higher School of Economics\\
  Usacheva St.\ 6\\
	119048 Moscow\\
	Russia \textit{and}
Skolkovo Institute for Science and Technology\\
 Skolkovo Innovation Center\\
 143026 Moscow\\
 Russia}
\email{sashaskrip@gmail.com}

\begin{abstract}
	We introduce a new renormalization procedure on double rotations, which is reminiscent of the classical Rauzy induction.
	Using this renormalization we prove that the set of parameters which induce infinite type double rotations has Hausdorff dimension strictly smaller than $3$.
	Moreover, we construct a natural invariant measure supported on these parameters and show that, with respect to this measure, almost all double rotations are uniquely ergodic.



\end{abstract}

\dedicatory{On the occasion of 80th anniversaries of V. I. Oseledets and A. M. Stepin}
\maketitle

\section{Introduction}
A \emph{double rotation} is a map $T\colon [0,1) \to [0,1)$ of the form:
\begin{equation}\label{eq:defdoublerotations}
	Ty =
	\begin{cases}
		y+\alpha  \pmod 1 & \text{if } y \in [0,c)  \\
		y+\beta \pmod 1   & \text{if } y \in [c,1).
	\end{cases}
\end{equation}

Double rotations, defined in~\cite{SuzukiItoAihara05}, are a class of \emph{interval translation mappings} (ITMs), see \cref{sec:history} for more details.
A double rotation is of finite type if its attractor is an interval and of infinite type if the closure of the attractor is a Cantor set.
It is easy to see that the restriction of a double rotation of finite type to its attractor is simply a rotation.
We prove

\begin{theorem}\label{main}
	The set of parameters that give rise to a double rotation of infinite type has Hausdorff dimension strictly less than $3$.
\end{theorem}

This theorem is a refinement of the statement obtained by H. Bruin and G. Clack in~\cite{BruinClack12} (see \cref{infinite} below) and agrees with their  numerical bounds.

The way in which we prove this result is significantly different from all the ideas described in~\cite{BruinClack12}.
Namely, we introduce a new renormalization procedure that is morally very close to the classical Rauzy induction for interval exchange transformations.
This renormalization allows us to define the interesting parameter set $X$ (the one that gives rise to the infinite type double rotations) as an attractor of some fractal set; by its nature our renormalization can be seen as a Markovian multidimensional fraction algorithm. Using the formalism introduced by  Fougeron in~\cite{Fougeron20}, we check that the \emph{simplicial system} $G$ associated with this algorithm is \emph{quickly escaping}.
This is an important technical condition that in some sense plays the role of bounded distortion in the classical Markov-Gibbs measure theory.
Thus we can apply results of~\cite{Fougeron20} on fractals sets to our setting.
In particular, the application of Theorem 1.5 of \cite{Fougeron20} results in our \cref{main}.
Moreover Theorem 4.24 in~\cite{Fougeron20} allows us to define a probability measure $\mu$, invariant with respect to our renormalization procedure, that induces the unique measure of maximal entropy on the canonical suspension of our parameter space $X$ (see Section~2.4 in \cite{Fougeron20} for definition of the suspension).
This produces a canonical invariant measure on the set of infinite type double rotations and enable us to describe the generic dynamical behaviors of such maps.

\begin{theorem}\label{unique}
Almost all double rotations of infinite type, with respect to the measure $\mu$, are uniquely ergodic.
\end{theorem}

\begin{remark}
	Bruin and Clack proved a similar statement for all measures invariant with respect to the SIA induction and supported on $X$, but the existence of any such a measure was up to now still unknown.
	The idea that thermodynamical formalism could be a tool to construct such a measure with support on $X$ was mentioned in~\cite{BruinClack12}; however, no explicit construction was provided.
\end{remark}

\subsection*{Acknowledgments}
A.~Skripchenko appreciates the support of RSF-ANR Grant, Project 20-41-09009.
We thank Serge Troubetzkoy for his numerous remarks and comments which greatly improved the paper.

\section{Historical remarks}\label{sec:history}
\subsection{Interval translation mappings}
We begin by giving the definition of \emph{interval translation mappings} or ITMs for short.

\begin{definition}
	An \emph{interval translation map} is a piecewise translation map $T$ defined on an half-open interval $I \subset \RR$ with values in $I$.
	We call $T$ a $n$-interval translation map (or $n$-ITM) if $I$ has $n$ maximal half-open sub-intervals to which the restriction of the $T$ is a translation.
	The endpoints of theses intervals are called \emph{singularities} of the map, and the endpoints of the images of the intervals are the images of the singularities.
\end{definition}

Two examples of ITMs on $3$ intervals are shown in \cref{3itm}.
ITMs were introduced in~\cite{BoshernitzanKornfeld95} as a natural generalization of \emph{interval exchange transformations} (IETs).
Unlike IETs these transformations are not necessarily surjective: the images of the intervals do not need to form a partition, they simply form a collection of subintervals of the original interval.

\input{tikz/example.tikz}

It was noticed already in~\cite{BoshernitzanKornfeld95} that each ITM is either of \emph{finite} or \emph{infinite} type.
This classification is based on the properties of the \emph{attractor} of the ITM. Namely, for a given mapping $T$ we consider the sequence $\Omega_n = I\cap TI\cap T^2I\cdots\cap T^{n}I$.
If this sequence stabilizes for some $N\in\NN$, i.e., $\Omega_k = \Omega_{k+1}$ for all $k\ge N$, the ITM $T$ is of \emph{finite type}.
If there is no such $N$ then closure of the limit set $\Omega =I\cap TI\cap T^2I\cdots$ is a Cantor set, see~\cite{SchmelingTroubetzkoy00}, and the ITM is said to be of \emph{infinite type}.

Dynamics of ITMs of finite type basically coincides with the one of IETs. However, ITMs of infinite type are remarkably different.
Boshernitzan and Kornfeld described the first example of ITM of infinite type.
In the same paper, they formulated the following

\begin{conjecture}\label{conjecture}
	The set of parameters that give rise to ITMs of infinite type has zero Lebesgue measure.
\end{conjecture}

They also stated a few interesting problems, including the question about the number of invariant measures for ITMs, their mixing and spectral properties, and their complexity (in the sense of symbolic dynamics).
Finally, they made the first attempt to describe ITMs in terms of substitutions.

Buzzi and Hubert showed that the number of invariant measures for ITM cannot exceed the number of intervals (see their Theorem in~\cite{BuzziHubert04}).
Their bound is optimal for ITMs (see~\cite{BruinTroubetzkoy03} and ~\cite{Bruin07}).

A few answers on the questions posed by Boshernitzan and Kornfeld (including Conjecture~\ref{conjecture}) were obtained only for a special class of ITMs, called \emph{double rotations}.
We outline the known results in the next section.

\subsection{Results for Double rotations}\label{DRhistory}
In~\cite{BoshernitzanKornfeld95} it was shown that ITMs on two intervals are always of finite type.
\emph{Double rotations}, introduced by Suzuki, Ito and Aihara in~\cite{SuzukiItoAihara05}, appeared to be the first nontrivial subclass of ITM and the only one for which, to the best of our knowledge, the Conjecture~\ref{conjecture} is proved.
In their paper, Suzuki, Ito and Aihara described a renormalization procedure for double rotations (known as \emph{SIA induction}).
An application of this procedure revealed a complicated fractal structure of the set of parameters that correspond to the maps of infinite type.
Using the same renormalization, Bruin and Clark proved in~\cite{BruinClack12} the following

\begin{theorem}\label{infinite}
	For Lebesgue almost every point in the parameter space the corresponding double rotation is of finite type.
\end{theorem}

As mentioned above, no explicit construction of the ergodic, invariant with respect to SIA induction, measure was suggested in~\cite{BruinClack12}; however, the following conditional theorem was established:

\begin{theorem}\label{uniqueergodic}
	For any measure $\mu$ that is invariant with respect to the SIA induction $S$, the set of parameters that give rise to non-uniquely ergodic double rotations of infinite type has zero measure.
\end{theorem}

The proof relies on symbolic dynamics: double rotations are described in terms of substitutions. Our proof of Theorem \ref{unique} is very different in spirit.

It follows from~\cite{BuzziHubert04} that double rotations cannot carry more than $2$ invariant measures.
Quite explicit examples of  double rotations with exactly two invariant measures were constructed by Bruin and Troubetzkoy in~\cite{BruinTroubetzkoy03} (see Theorem $11$ in their paper).
In this paper the authors worked with a special subclass of double rotations.
They constructed a renormalization map which is slightly different from the one in~\cite{SuzukiItoAihara05} and used it to prove a weaker version of Theorem~\ref{infinite} (see their Theorem $6$).
This result was later generalized by Bruin for a slightly more extended subclass of ITMs (see Theorem 1 in~\cite{Bruin07}).

Finally, in~\cite{SkripchenkoTroubetzkoy15}, the authors focused on a different part of the program suggested by Boshernitzan and Kornfeld.
Namely, they studied billiards in rational polygons with so called spy mirrors and worked on the questions about complexity of the trajectories of these billiards.

\subsection{Other results}
In~\cite{Volk14} Volk extended Theorem~\ref{infinite} to the case of ITM on three intervals:

\begin{theorem}
	Lebesgue almost every ITM on three intervals is of finite type.
\end{theorem}

This results follows from the fact that every ITM on three intervals can be reduced to a rotation or to a double rotation.

The orbit structure of ITMs was studied in~\cite{SchmelingTroubetzkoy00} where the sharp upper bound on the number of minimal sets of such a mapping is obtained.

\section{The \zorich-induction}

In the current section we introduce a new renormalization procedure for double rotations.
Double rotations (given by ~\eqref{eq:defdoublerotations}) will be considered as a subclass of $3$-ITMs.

We represent in \cref{3itm} two examples of $3$-ITMs.
Notice that it is enough to represent the interval domains and their images since this will define uniquely the translations.

For any $n$-ITM $T$ there is a subinterval such that the first return map of $T$ on this subinterval is an ITM with no gap at the extremal sides.
Hence, we can restrict our study to ITMs with no gaps at the extremal sides as in \cref{3itm}.

\input{tikz/reduction.tikz}

\begin{lemma} \label{reduction}
	Let $T$ be a $3$-ITM with no gaps at on the extremal sides and a singularity belonging to a gap in the image.
	Then there exists a subinterval on which the first return map of $T$ is a rotation.
\end{lemma}

\begin{proof}
	Let $T$ be an ITM as in the statement of the Lemma, such as the one on top of \cref{fig:red}.
	We consider the induced transformation on $I\cap TI \cap T^2 I$, as represented in the middle drawing of \cref{fig:red}.
	Since one of the singularities is outside $TI$, after inducing on $I\cap TI\cap T^2 I$ we are left with only two intervals, possibly overlapping.
	If there is no overlap the transformation is already a rotation.
	Otherwise, there is a gap on at least one of the sides of the ITM.
	This gap is as big as the overlap between the two intervals.
	By considering the transformation induced on the image, we get an ITM without gaps at the sides.
	This guarantees that the overlap has vanished.
	In other words, the final transformation we have obtained is a rotation.
\end{proof}

\begin{convention}\label{noextremalgaps}
Unless explicitly mentioned, we will always assume that ITMs do not have gaps on the extremal sides and they act on the interval $I=[0,1)$.
\end{convention}

We now give a description of $3$-ITMs which mimics the Interval Exchange Transformations case, following our \cref{noextremalgaps}.

\begin{definition}\label{def:3ITM}
	A $3$-ITM can be described on an alphabet $\alphabet$ by
	\begin{itemize}
		\item two bijections $\pi_0$, $\pi_1\colon \alphabet \mapsto \{1, 2, 3\}$
		\item a length vector $\lambda_\alphabet \in \RR_+^\alphabet$,
		\item a translation parameter $t \in \RR_+$, measuring the displacement of the second interval.
	\end{itemize}
	Given an $\alpha\in\alphabet$ and a point
	\[
		x \in \left[\sum_{\pi_0(\beta) < \pi_0(\alpha)} \lambda_{\beta}, \sum_{\pi_0(\beta) \le \pi_0(\alpha)} \lambda_{\beta} \right]
	\]
	the ITM is defined by
	\[
		T(x) = x - \sum_{\pi_0(\beta) < \pi_0(\alpha)} \lambda_{\beta} + t,
	\]
	if $\pi_1(\alpha) = 2$, and
	\[
		T(x) = x - \sum_{\pi_0(\beta) < \pi_0(\alpha)} \lambda_{\beta} + \sum_{\pi_1(\beta) < \pi_1(\alpha)} \lambda_{\beta}.
	\]
	otherwise.
	In the previous formulae, an empty sum corresponds to $0$.
\end{definition}

From now on, we will always assume that the $3$-ITM is \emph{irreducible}.
This means that $\pi_1\circ\pi_0^{-1}(\{1\})\neq\{1\}$ and $\pi_1\circ\pi_0^{-1}(\{1,2\})\neq\{1,2\}$.
This is a natural assumption, since reducible $3$-ITMs can be reduced to $2$-ITMs, which are of finite type, see~\cite{BoshernitzanKornfeld95}.
Notice that the permutation $\pi_1=(3,1,2)$ yields a $3$-ITM with a gap on the extremal side, contradicting our \cref{noextremalgaps} and it is thus forbidden.

We now define our first renormalization procedure on $3$-ITMs, similar to the Rauzy-Veech induction for interval exchange transformations.
In order not to increase the number of intervals, we have to handle carefully the overlap in the image.
This induction is rather natural and easy to define; unfortunately, its dynamical properties are hard to describe.
So, in the next section we will define another induction that does not have this issue and then show that both of them are simply two versions of the same process.
More precisely, we can accelerate the induction we will introduce in the next section in order to recover the \zorich-induction.

\input{tikz/rauzy_diagram.tikz}

\begin{definition}[\zorich-induction]\label{def:zorich}
	Let $T$ be a $3$-ITM with the overlap being on the left hand side of the interval $I$.
	The right \emph{$\mathcal{Z}$-induction} of $T$ consists in taking the first return map on a subinterval $I'$ of $I$ starting at the left end point of $I$ chosen as follows:
	\begin{itemize}
		\item If $\lambda_3>\lambda_1$ we choose $\widetilde{I}=[0,\lambda_3-\lambda_1)$. If $\widetilde{I}$ has no gap on the right-hand side, we define $I'=\widetilde{I}$.
		Otherwise we define $I'=\widetilde{I}\cap T\widetilde{I}$.
		In this case we say that top is \emph{winning} and that bottom \emph{losing}.
		\item If $\lambda_3<\lambda_1$ we choose $I'=[0,\lambda_1+\lambda_2)$.
		In this case we say that bottom is \emph{winning} and that top \emph{losing}.
	\end{itemize}

	If the overlap happens on the right hand side, we conjugate with a flip.
	More precisely, we apply the reflection around $1/2$, apply the right induction and then apply the reflection around the midpoint of the interval obtained after the induction.
\end{definition}

Since we are inducing on an interval of the form $[0,a)$, obtained by cutting the original interval on the right, we call the previous induction a \emph{right} induction.

In \cref{fig:zdiagram}, we represent the Rauzy diagram for the \zorich-induction.
This diagram is defined for the right induction.
Some steps will induce an ITM for which the overlapping intervals are on the right.
In this case we apply the around the midpoint of the interval obtained after the induction.
These cases are marked with a tilde.

\input{tikz/first_combinatoric.tikz}

\input{tikz/second_combinatoric.tikz}

We give a complete description of the combinatorial cases that can happen in \cref{fig:z-induction,fig:z-induction_rot}.
%In the case where the rightmost interval on top is larger than the one at bottom, we say that top is \emph{winning} and bottom is \emph{losing}.
%Similarly, if the rightmost interval on the top is smaller than the one at the bottom, we say that bottom is winning and top is losing.
The pictures are labeled with $t$ if the top is winning and $b$ if the bottom is winning.

We remark that the case where a singularity on the top interval is in the gap does not appear in the list.
This is due to the fact stated in \cref{reduction} that they induce a rotation on a subinterval.
In this cases, we stop the induction.

\section{The \rauzy-induction}
In this section we construct our renormalization procedure and show that there is a way to associate a simplicial system to it.
\subsection{A different description of ITMs}

As we said above, the \zorich-induction is convenient to show that the number of intervals remains stable, but its dynamical properties are hard to describe.
Let us comment a bit more on this point.
The $\mathcal{Z}$-induction defines a matrix $Z$ that expresses the new length vector $\lambda\in\RR_+^\alphabet$ and the new translation parameter $t$, as defined in \cref{def:3ITM}, in terms of the old ones.
One would like this matrix to have non-negative coefficients in order to apply classical arguments, analogous to the Perron-Frobenius theorem.
Unfortunately, this is not the case.
The presence of negative coefficients in this matrix seems to be caused by the translation parameter $t$.
To get rid of this problem, in this section we give an alternative representation of ITMs which is combinatorially more involved but allows to describe the transformation $T$ solely in terms of its length vector $\lambda\in\RR_+^\alphabet$.
%It will turn out that this new induction is a slower version of the \zorich-induction.

Given a $3$-ITM satisfying \cref{noextremalgaps}, we consider the two overlapping intervals in the image and cut them into two pieces each: one coinciding with the overlap and one for the remaining part of each interval.
In this way we obtain five intervals; two of them have, by construction, the same size.
An example of this splitting procedure is shown in \cref{fig:split}.

\input{tikz/split.tikz}

Let us label the five intervals we have obtained with letters in an alphabet $\alphabet$ and let $w_0$ be the word formed by them.
Then $\alphabet$ has $4$ letters and $w_0$ has exactly one letter repeated twice.
Since ITMs act by translations, the length of an interval in the domain and the length of its image are the same.
Hence, the gap in the image has precisely the length of one of the five beginning intervals, which in turn is the length of the overlap in the image of the ITM.
If the label of this interval is $\alpha$, let us denote the gap with the label $\underline{\alpha}$.
We then form a word $w_1$ with letters in $\alphabet\cup\underline{\alphabet}$, describing the position of the intervals after applying the ITM.
Following the classical way to represent IETs, we write the two words one above another.
Here are two examples corresponding to \cref{fig:split}
\[
	\begin{pmatrix}
		A & D & B & C             & D \\
		C & D & B & \underline{D} & A
	\end{pmatrix},
	\qquad \qquad
	\begin{pmatrix}
		A & B             & C & B & D \\
		C & B & D & \underline{B} & A
	\end{pmatrix}.
\]

Let $\sigma$ be the map that associates to an integer $k$, with $1 \le k \le n$, the unique position in $w_1$ at which the letter $w_0(k)$ appears.

We call the pair $\left(\begin{smallmatrix} w_0\\ w_1\end{smallmatrix}\right)$ an \emph{ITM permutation}.
Let us consider such a permutation and a length vector $\lambda \in \RR_+^\alphabet$.
Given a word $w$ of length $|w|=n$ we define
\[
	\lambda(w) = \sum_{i=1}^n \lambda_{w(i)}.
\]
Let $w^{(k)}$ denote the prefix of length $k-1$ of $w$, i.e., the subword before the $k^\text{th}$ letter of $w$, for $1\le k \le n$, with $w^{(1)}$ being the empty word.
Given a point
\[
	x \in \left[\lambda\left(w_0^{(k)}\right), \lambda\left(w_0^{(k+1)}\right)\right],
\]
the ITM $T$ is the map defined by
\[
	F(x) = x + \lambda\left(w_1^{\left(\sigma(k)\right)}\right) - \lambda\left(w_0^{(k)}\right).
\]
As usual, an empty sum is defined to be $0$.

Notice that, with this new presentation, there is no need for a translation vector $t$ as in \cref{def:3ITM}.
The price we have to pay is the increased combinatorial complexity.

\subsection{Definition of the \rauzy-induction}
We start by defining a direct generalization of the classical left and right Rauzy induction on interval exchange transformations to our setting.

\begin{definition}[Rauzy induction]
	The \emph{right Rauzy induction} associates to an ITM defined on $I = [0,1)$ the first return map to the subinterval $[0, x)$ where $x$ is the rightmost singularity or image of a singularity in the interior of $I$.

	The \emph{left Rauzy induction} is defined symmetrically.
\end{definition}

We remark that the Rauzy induction yields in some of these case ITMs which have a gap on the right side, thus violating our \cref{noextremalgaps}, see \cref{fig:RauzyTwice} for an example of this.
Moreover, this induction induces a new ITM but, contrary to the case of IETs, it may increase the number of intervals of continuity.
Specifically, this happens in the cases where the winning interval appears twice in the combinatorial word.
To get around this problem, we will use the following choice of left and right induction, depending on the position of the gap.

\begin{definition}[\rauzy-induction]
	We define the \rauzy-induction as follows:
	\begin{itemize}
		\item if the hole is in one of the last two intervals of $w_1$, we apply the right Rauzy induction;
		\item if the hole is one of the two first intervals of $w_1$, we apply the left Rauzy induction.
	\end{itemize}
\end{definition}

\begin{remark}
	Another option is to always apply the right Rauzy induction and flip the interval $I$ if the hole is not in the last two positions of $w_1$.
	The two choices produce the same graph.
	Hence, we will switch freely from one to the other depending on which one is more convenient for the task at hand.
\end{remark}

Our first goal in this section is to show the following Lemma which informally says that the \zorich-induction is an accelerated version of the \rauzy-induction.

\begin{lemma}\label{thm:acceleration}
	If no singularity of a $3$-ITM belongs to the gap, there exist a finite number of steps of the \rauzy-induction which give the same $3$-ITM as the right \zorich-induction.
\end{lemma}

Before giving the proof of this result, let us give a more detailed description of the right Rauzy induction for an irreducible ITM in which the gap interval appears in $w_1$ at the last or penultimate position.
A similar description, obtained by applying a central symmetry to the intervals, could be given for the left Rauzy induction for an irreducible ITM in which the gap appears in the first two positions of $w_1$.
We will later show that these cases are indeed the only ones appearing in the \rauzy-induction as in the assumption of the Lemma.
Notice that the induction will induce ITM permutations that always have $5$ intervals in $w_0$ and $w_1$ but do not necessarily correspond to 3-ITMs.

We recall that a \emph{substitution} $\tau$ is an application from an alphabet $\alphabet$ to the set of nonempty finite words on the $\alphabet$, which can be extended to a morphism of $\alphabet^*$ by concatenation: $\tau(WW')=\tau(W)\tau(W')$, see, e.g.,~\cite{Pytheas}.

\begin{definition}\label{def:rrauzy}
	Consider an irreducible ITM with 5 letters in each words $w_0$ and $w_1$.
	Assume it has a single gap that appears as one of the last two intervals in $w_1$.
	Then we have another description of the right Rauzy induction which goes as follows.

	Start by comparing the rightmost interval in the domain with the rightmost interval in the image.
	In other words, we compare the lengths of $\lambda_{w_0(5)}$ and $\lambda_{w_1(5)}$.
	We put aside the case when $\lambda_{w_0}(5)=\lambda_{w_1}(5)$ for which it is clear that the double rotation reduces to a rotation.
	We then say that the larger one is the \emph{winning} label and the smaller interval is the \emph{losing} one.
	Let $\epsilon$ be equal to $0$ if the top interval is winning and $1$ otherwise, then $a:=w_\epsilon(5)$ is the label of the winning interval and $b:=w_{1-\epsilon}(5)$ is the losing one.
	Let $s$ be the substitution defined by $a \mapsto a \cdot b$.
	The \emph{right Rauzy induction} acts on the combinatorics of ITMs as follows, where $\hat w$ denotes the word $w$ after induction:\\

	\noindent$\bullet$ if $a$ occurs exactly once in each of the two words
		      \[
			      \begin{split}
				      \hat{w}_\epsilon &= w_\epsilon\\
				      \hat{w}_{1-\epsilon} &= s\left(w_{1-\epsilon}(1, \dots, n-1)\right),
			      \end{split}
		      \]
	see \cref{fig:RauzyOnce}.
	Notice that in this case the losing interval could be the hole.\\

  \noindent$\bullet$ if $a$ occurs twice in $w_0$ and the intervals are in the following setting
		      \[
			      \begin{pmatrix}
				      \dots &              & a \\
				      \dots & \underline a & b
			      \end{pmatrix}
		      \]
		      then,
		      \[
			      \begin{split}
				      \hat{w}_0 &= s\left(w_0 \left(1, \dots, n-1\right)\right) \\
				      \hat{w}_1 &= s\left(w_1 \left(1, \dots, n-2\right)\right) \cdot \, \underline b,
			      \end{split}
		      \]
	see \cref{fig:RauzyTwice}.\\

  \noindent$\bullet$ if the winning interval is to bottom one corresponding to the gap, the initial 3-ITM has a singularity which belongs to a gap in the image.
	Hence, according to \cref{reduction}, the induced ITM is a rotation on a subinterval.
	In this case we stop the induction.

	Moreover, in each of the first two cases, the Rauzy induction acts on length vector by
	\[
		\begin{split}
			\hat{\lambda}_a &= \lambda_a - \lambda_b \\
			\hat{\lambda}_b &= \lambda_b.
		\end{split}
	\]
\end{definition}

\input{tikz/rauzy.tikz}

The moves of the Rauzy induction in all these cases can be represented as in~\cite{Fougeron20} by a simplicial system\footnote{See \cref{sec:simplicial} for the definition of simplicial system.} associated to a labeled graph whose edges are described for a given ITM permutation in \cref{fig:graphG}.
The labels in the graph $G$ correspond to the losing interval.

\begin{figure}[!h]
	\begin{subfigure}{.8\textwidth}
		\begin{tikzpicture}[shorten >=1pt,node distance=4cm,on grid,auto]
			\node (sigma) {$\begin{pmatrix} v_0 \ a \\ v_1 \ b \end{pmatrix} $};
			\node (nsigma) [left= of sigma] {$\begin{pmatrix} s\left( v_0\right) \\  v_1 \ b \end{pmatrix}$};
			\node (1sigma) [right= of sigma] {$\begin{pmatrix}  v_0 \ a\\ s\left( v_1\right) \end{pmatrix} $};

			\path[->] (sigma) edge node[above] {$b$} (1sigma);
			\path[->] (sigma) edge node[above] {$a$} (nsigma);
		\end{tikzpicture}
	\end{subfigure}
	\begin{subfigure}{.8\textwidth}
		\begin{tikzpicture}[shorten >=1pt,node distance=4cm,on grid,auto]
			\node (sigma) {$\begin{pmatrix}  v_0 \ \ a \\  v_1 \ \underline a \ b\end{pmatrix} $};
			\node (nsigma) [right= of sigma] {$\begin{pmatrix} s\left( v_0\right) \\ s\left( v_1\right) \ \underline b \end{pmatrix} $};
			\node (1sigma) [left= of sigma] {$\begin{pmatrix} s\left(  v_0\right) \\  v_1 \ \underline a \ b \end{pmatrix}$};

			\path[->] (sigma) edge node[above] {$a$} (1sigma);
			\path[->] (sigma) edge node[above] {$b$} (nsigma);
		\end{tikzpicture}
	\end{subfigure}
	\begin{subfigure}{.8\textwidth}
		\begin{tikzpicture}[shorten >=1pt,node distance=4cm,on grid,auto]
			\node (sigma) {$\begin{pmatrix}  v_0 \ a \ b \\  v_1 \ \ \underline a \end{pmatrix} $};
			\node (nsigma) [right= of sigma] {rotation};
			\node (1sigma) [left= of sigma] {$\begin{pmatrix}  v_0 \ a \ b \\ s\left( v_1\right) \end{pmatrix}$};

			\path[->] (sigma) edge node[above] {$a$} (1sigma);
			\path[->] (sigma) edge node[above] {$b$} (nsigma);
		\end{tikzpicture}
	\end{subfigure}
	\caption{The graph $G$ associated to the right \rauzy-induction.}
	\label{fig:graphG}
\end{figure}

We now prove \cref{thm:acceleration}, showing at the same time that the Rauzy induction remains in the specific cases of \cref{def:rrauzy}.
\begin{proof}[Proof of \cref{thm:acceleration}]
	We assume, without loss of generality, that $I = [0,1)$.
	Both the \rauzy-induction and the \zorich-induction associate to the initial ITM its first return map on a subinterval.
	Let $[0,s_r)$ the interval corresponding to the \rauzy-induction and $[0,s_z)$ the one corresponding to the \zorich-induction.
	We denote by $[0, s_r^{(n)})$ the associated subinterval after $n$ iterations of the \rauzy-induction where the right Rauzy induction is applied.
	It is clear from the definitions that $s_z \le s_r$.
	We show in what follows that for some $n \geq 1$, $s_z = s_r^{(n)}$.

	Let us denote with $J^1, J^2, J^3$ the intervals in the domain of the $3$-ITM and with $J_1, J_2, J_3$ their images.
	Similarly, we denote $\{I^\alpha\}_{\alpha\in\alphabet}$ the intervals obtained after the splitting procedure which produces an ITM permutation from the original $3$-ITM and $\{I_\alpha\}_{\alpha\in\alphabet}$ their images.
	Since no singularity belongs to the gap, $J_3=I_\alpha$, for some $\alpha\in\alphabet$, i.e., $J_3$ is not subdivided when we represent the $3$-ITM as an ITM permutation on $5$ intervals.
	Hence, the bottom word in the ITM permutation ends by $\underline \gamma \cdot \alpha$ with $\alpha,\gamma \in \alphabet$.

	We start by considering the cases in which the bottom interval wins in the \zorich-induction: $|J_1|>|J^3|$, which corresponds to cases labeled by b in \cref{fig:z-induction,fig:z-induction_rot}.
	In this case, either $J^3=I^\alpha\cup I^\beta$ or $J^3=I^\alpha\cup I^\beta \cup I^\gamma$, for some $\alpha, \beta,\gamma\in\alphabet$.
	Thus the \zorich-induction corresponds to two or three step of the \rauzy-induction.

	We now consider the case in which the top interval wins in the \zorich-induction: $|J^3|>|J_1|$.
	In these cases, $s_z$ is the last image of a singularity before the gap.
	Thus we only need to show that the gap will lose after a finite number of steps of the \rauzy-induction.
	We have
	\[
		J^3=I^\alpha \cup I^\gamma \cup I^\beta, \quad J^3=I^\gamma \cup I^\alpha \quad \text{or} \quad J^3=I^\alpha \cup I^\gamma
	\]
	and the rightmost bottom interval $J_1$ together with the gap is equal to $I_{\underline \gamma} \cup I_\eta$ where $\gamma$ is the label of the gap and the labels $\alpha, \beta, \gamma$, and $\eta$ are distinct.
	Moreover, the left singularity of the interval $J^3$ has to be to the left of $s_z$, which implies that the left singularity of the corresponding interval $I^\alpha$ or $I^\gamma$ is left of $s_z$.

	In the \rauzy-induction, any of the labels $\alpha, \beta, \gamma$, and $\eta$ can win.
	However, since the corresponding interval in the image are all to the left of $s_z$ this will not increase the number of singularities in $(s_z,1)$.
	In this setting, since we are in the first two cases of \cref{def:rrauzy}, the \rauzy-induction is always defined and does not stop since we never end up with the gap interval winning.
	In fact, the first case of the induction will keep the gap interval with at most one interval on its right, and the induction remains well defined in the second case.
	After a finite number of cases where the winner appears only once in each word and possibly one case where it appears twice, the bottom interval will be the gap.
	As noticed previously, by assumption, the interval on the top word will still be larger than the bottom one.
	Thus the gap will lose and at this step $s_z = s_r^{(n)}$ as we wanted.
\end{proof}

It follows from the previous result that the labeled graph $G$ in \cref{fig:graphG} is well defined.

\begin{lemma}\label{lemma:win}
	For any infinite sequence of well defined left and right Rauzy inductions on an ITM, every letter wins and loses infinitely many times.
\end{lemma}

\begin{proof}
	We begin by observing that if a letter does not win infinitely many times, then the length of the corresponding interval is bounded from below.
	Thus, it can lose only a finite number of times.
	Up to applying left and right Rauzy induction a finite number of times, we can assume that there is a set $\cL$ of letters that lose infinitely many times and such that no other letter neither wins or loses.
	We denote its complementary set by $\cL^c$.

	We claim that all the letters in $w_0$ and $w_1$ belonging to $\cL^c$ form a single subword.
	In other words, for any $\epsilon \in \{0,1\}$ there exists a decomposition $w_\epsilon = r \cdot v \cdot t$ such that $v$ is a word in alphabet $\cL^c$ and $r,t$ are in $\cL$.
	Notice that the words $r$ and $t$ may be empty.

	To show our claim, assume that there is a decomposition $r \cdot v_1 \cdot s \cdot v_2 \cdot t$ such that $v_1, v_2$ are words with letters in $\cL^c$ and $r,s,t$ are words with letters in $\cL$, and $|s| > 0$.
	Since the intervals around $s$ never lose, the intervals in $s$ also will not lose.
	Hence, the number of letters in $s$ is not decreasing.
	By definition of $\cL$, the letters of $s$ have to win eventually.
	Whenever this happens, the losing letter is added to to $s$, thus increasing the length of $s$ itself.
	Since there number of letters is finite and the length of $s$ keeps increasing, we have a contradiction.

	We now show that all letters winning infinitely many times also lose infinitely many times.
	Indeed, imagine that a letter $\alpha$ wins infinitely many times but never loses.
	Hence either $I^\alpha$ or $I_\alpha$ are at the beginning or end of the interval $I$.
	By symmetry of the problem, we can assume that $\alpha$
	will remain in the first position of $w_0$ or $w_1$.
	The other word can be decomposed in $r \cdot \alpha \cdot s$, where $r$ is not the empty word.
	By assumption, the letters in $r$ always lose in the left \rauzy-induction	against $\alpha$.
	Thus, the length of the word $r$ is non-increasing.
	Since $\alpha$ is winning, the corresponding interval is getting shorter.
	Eventually, it will be shorter that the length of the interval labeled by the first letter of $r$ and hence lose to this same interval in the left \rauzy-induction, which is a contradiction.

	To sum up, we have two words on top and bottom with a central part formed by the letters that never win.
	This central part of $w_0$ and $w_1$ is invariant by the left and right Rauzy induction.
	As all the other letters lose infinitely many	times, the lengths of the corresponding intervals go to zero.
	Since these central parts have exactly the same position in the top and bottom intervals, the left and right Rauzy induction stops after a finite number of times, contradicting our assumption.
\end{proof}

\subsection{Simplicial systems}\label{sec:simplicial}
In this section we give, following~\cite{Fougeron20}, the definition of simplicial system.

Let $G=(V,E)$ a graph labeled on an alphabet $\alphabet$ by a function $l\colon E\to\alphabet$.
We require that, for all $v\in V$, the restriction of $l$ to the edges starting at $v$ in injective.
Let $\|\cdot\|$ be the norm on $\RR_+^\alphabet$ given by $\|\lambda\|=\sum_{\alpha\in\alphabet}\lambda_\alpha$.
Consider the simplex of dimension $|\alphabet|-1$ defined by $\Delta=\{\lambda\in\RR_+^\alphabet : \|\lambda\|=1\}$.

Given a vertex $v\in V$ let $v_{\text{out}}$ be the set of all edges going out of $v$ in $G$.
We define a partition of $\Delta$ whose atoms, for an edge $e\in E$, are given by
\[
	\Delta^e=\{ \lambda\in\Delta : \lambda_{l(e)} < \lambda_\alpha, \text{ for all } \alpha\in l(v_{\text{out}}) \text{ and } \alpha\neq l(e) \}.
\]
To the edge $e$ we associate the matrix
\[
	M_e = \operatorname{Id} + \sum_{\substack{\alpha\in l(v_{\text{out}}) \\ \alpha\neq e}} E_{\alpha,l(e)},
\]
where $E_{a,b}$ is the elementary matrix with coefficient $1$ at the row $a$ and column $b$ and all the other coefficients equal to $0$.

Define $\Delta^G=V\times\Delta$ and define the map $T\colon\Delta^G\to\Delta^G$ as follows.
For all $\lambda\in\Delta^e$ with the edge $e$ starting at $v$ and ending at $v'$, we say
\[
	T(v,\lambda)=(v', T_e(\lambda)),
\]
with
\[
	\begin{split}
		T_e\colon\Delta^e &\to\Delta\\
			\lambda &\mapsto \frac{M_e^{-1}\lambda}{\|M_e^{-1}\lambda\|}.
	\end{split}
\]
We say that $G$ is a \emph{simplicial system} and the map $T$ is its associated \emph{win-lose induction}.

From these definitions it is clear that the projectivization of the right \rauzy-induction is the win-lose induction associated to the simplicial system defined by the graph $G$ in \cref{fig:graphG}.

\subsection{Proof of the main results}
In order to prove our main results we need another definition from~\cite{Fougeron20}.

Given a simplicial system $G$ and $\cL\subset\alphabet$, we denote by $G_\cL$ the subgraph of $G$ with the same set of vertices $V$ and whose edges are as follows.
For a vertex $v\in V$:
\begin{itemize}
	\item if at least one outgoing edge from $v$ is labeled in $\cL$ then the set of outgoing edges in $G_\cL$ is
	\[
		v_\text{out}^\cL = \{e\in v_{\text{out}} : l(e)\in\cL\};
	\]
	\item otherwise $v_{\text{out}}^\cL=v_{\text{out}}$.
\end{itemize}

\begin{definition}[Strongly non-degenerating]
	\label{def:deg}
	We say that a simplicial system is \emph{strongly non-degenerating} if
	\begin{enumerate}
		\item for all vertices every letter wins and loses in almost every path with respect to Lebesgue measure,
		\item for all $\emptyset\subsetneq\cL\subsetneq\alphabet$ and all vertex $v$ in a strongly connected components $\mathscr{C}$ of $G_\cL$ then either
			\begin{itemize}
				\item $|l(v_\text{out})\cap\cL|\leq 1$;
				\item or there is a path from $v$ in $G$, labeled in $\cL$, which leaves $\mathscr{C}$.
			\end{itemize}
	\end{enumerate}
\end{definition}

%By using similarities of our algorithm with Rauzy induction, we obtain \cref{lemma:win} which implies the first part of the strongly non-degenerating property.
%An important part of the work in~\cite{Fougeron20} is done to relax this condition in order to apply it to more general systems (\textit{e.g.} multidimensional continued fraction algorithms) but is not required in our setting.

We are now ready to prove our main technical statement: we have to check that the simplicial system $G$ that we have defined above is quickly escaping.
The formal definition of \emph{quick escaping property} for a simplicial system was given in~\cite{Fougeron20} (see Section 3.1).
This property is a version of bounded distortion property: morally, it means that trajectories will go out of degenerate subgraphs in a small time with high probability.
Since stating formally the quick escaping property would require us to introduce a lot of the machinery from~\cite{Fougeron20}, we will refrain from doing so.

We will use the following result

\begin{theorem}[Theorems 3.13 of~\cite{Fougeron20}]\label{thm:nondegenerating}
	A strongly non-degenerating simplicial system is quickly escaping.
\end{theorem}

Let $F$ be the subgraph of $G$ obtained by removing the vertices corresponding to rotations and the edges pointing to them.

\begin{lemma}\label{quicklyescaping}
	The simplicial system associated to $F$ is quickly escaping.
\end{lemma}

\begin{proof}
By \cref{thm:nondegenerating} it is enough to check that $F$ is strongly non-degenerating.
The first part of \cref{def:deg} is proved in \cref{lemma:win}.
Thus we need to prove only the second part.

We remark that in the Rauzy induction a winning interval always keeps its position at one end of the interval $I$ where the ITM acts.
Let $\cL$ be a subset of letters as in \cref{def:deg}.
In $F_\cL$, by definition, a letter in $\cL$ always lose against any letter in the complementary set $\cL^c$.
Hence if an interval at the left or right extremity of $I$ is labeled in $\cL^c$ it will remain in the same position.
In particular if there is at least one interval labeled in $\cL$ on each side of $I$ then each point of the corresponding strongly connected component of $F_\cL$ satisfies the first case of property (2).
This also show that the number of extremities with at least on interval labeled in $\cL^c$ is constant in a strongly connected component of $F_\cL$.

If none of the extremal intervals of $I$ is labeled in $\cL^c$ by the first part of \cref{def:deg} there exists a finite path with only the last label in $\cL^c$.
Hence there exists a path labeled in $\cL$ that leaves a connected component with no extremal interval labeled in $\cL^c$.
This shows that the second part of (2) holds.

Finally, we deal with the case where there is only one extremal interval labeled in $\cL^c$ but the induction is comparing two interval labeled in $\cL$.
Notice that for the \zorich-induction, the interval that has no overlap in the image is always at the opposite position of its preimage.
Namely, if it is at the extreme right its preimage is at the extreme left and vice versa.
By \cref{thm:acceleration}, the same is true after a finite number of step for the split ITM, creating a \textit{big loop} for the Rauzy induction.
More precisely, there is a loop along which this letter wins against every other letter in the ITM.
This implies that one can make any letter in $\cL^c$ appear as the label of an extremal interval on this side, hence going to a strongly connected component with two sides having an interval labeled in $\cL^c$.
Since we have already taken care of this case, we have completed the proof.
\end{proof}

The set of lengths vectors for which the right \rauzy-induction does not stop in finite time starting from a given vertex of the simplicial system form a fractal subset (denoted by $\Delta(F)$ in \cite{Fougeron20}) of the simplex set.
Theorem 1.5 of~\cite{Fougeron20} then immediately implies the following.

\begin{corollary}
	The set of lengths vectors for which the right \rauzy-induction does not stop in finite time on $3$-ITMs is of Hausdorff dimension strictly smaller than $3$.
\end{corollary}

This corollary is equivalent to \cref{main}.
Finally we prove our result about uniquely ergodic double rotation, \cref{unique}.

\begin{proof}[Proof of \cref{unique}]
	Theorem 4.24 of~\cite{Fougeron20} guarantees the existence of a measure $\mu$ which is invariant under the right \rauzy-induction and which induces the unique measure of maximal entropy on the canonical suspension.
	As explained in Section~4.1 of~\cite{Fougeron20}, the quickly escaping property enables to construct a uniformly expanding acceleration of the induction.
	Hence, the set of parameters which follow the same path, generic for $\mu$, is a single point in $X$.
	As in Theorem~5 of~\cite{BruinClack12}, we can now apply the criterion in Lemma 17 of~\cite{BruinTroubetzkoy03} to prove that almost every double rotation w.r.t.\ $\mu$ is uniquely ergodic.
\end{proof}

\section{Open questions}
In the current section we briefly discuss several open questions on double rotations and their generalizations that we consider challenging and interesting.

\subsection{Geometrical interpretation}
It is known that in the case of interval exchange transformations (IETs) there exists a duality between measured foliations on surfaces and IETs: every measured foliation identifies a family of IETs as the first return maps on some transversal, and for any IET there is a canonical way to construct a surface with a measured foliation on it such that the starting IET is a first return map on some transversal of the measured foliation we constructed.
This procedure is called \emph{suspension} and was described by W.~Veech in~\cite{Veech82}.
Suspensions became an efficient tool to understand ergodic properties of IETs and have also been widely used in study of geometry of moduli spaces (see, for example, \cite{Yoccoz10}).

Another close relative of double rotations (and, more generally, ITMs) is called \emph{systems of isometries} (see~\cite{GaboriauLevittPaulin94}).
Systems of isometries generalize the idea of ITMs in the following sense: we consider a finite set of identifications of subintervals of a given interval (or even multi-interval) via orientation-preserving isometries.
Systems of isometries appeared in geometric group theory as a tool to describe the action of a free group on $\mathbb R$-tree, see~\cite{GaboriauLevittPaulin94}.
There is a canonical way to define a suspension for any system of isometry (the corresponding object is also known as \emph{band complex}) and there exists a way to associate this foliated band complex with a measured foliation on a handle body (see~\cite{DynnikovHubertSkripchenko20}).

It would be interesting and useful to construct a geometrically meaningful suspension for double rotations (or for a wider class of ITMs).

Another problem that is important to solve in this context is the question of the existence of a ``natural'' continuous flow such that the suspension flow based on our induction procedure is its discretization, in a way analogue of Teichm\"uller flow and Veech flow for interval exchange transformations.
If there is a way to find such a flow, one could also think about the associated moduli spaces and their geometry.

\subsection{Connection with billiards}
It was shown in~\cite{SkripchenkoTroubetzkoy15} (using ideas from~\cite{BoshernitzanKornfeld95}) that double rotations can be considered as the first return map of the billiard flow for the so-called \emph{billiards with spy mirrors}.
This implies the following natural question, partly related to the problem of geometrical interpretation described above: is there a way to estimate the diffusion rate of the billiard trajectory in terms of the Lyapunov spectrum of the induction cocycle?
In such a case it would be interesting to investigate some properties of the Lyapunov spectra, such as simplicity and/or the Pisot property.

\subsection{Ergodic properties}
Once we know that the typical double rotation is uniquely ergodic, it is natural to ask about stronger chaotic properties, such as weak and strong mixing.
It is announced in a fresh preprint by S.~Troubetzkoy in ~\cite{Troubetzkoy21} that all double rotations admit non weakly-mixing invariant measure.
Thus, since almost all double rotations of infinite type are uniquely ergodic, almost all of them are not weakly mixing.
Hence a more challenging question is to find double rotations with exceptional behavior.
Namely, those that are minimal and of infinite type and are weakly-mixing.

For general ITMs, even the question about ergodicity and uniqueness of invariant measure remains open.

\subsection{Questions of complexity and combinatorial properties}
Some study of the complexity in connection with double rotations was performed in~\cite{SkripchenkoTroubetzkoy15} and~\cite{CassaigneNicolas10}.
More in detail, in~\cite{CassaigneNicolas10} it was shown that the subword complexity of the subshift associated with a special subclass of double rotations that corresponds to a subspace of codimension 1 is linear: $n+1 \le p(n)\le 3n$.
In~\cite{SkripchenkoTroubetzkoy15} some estimations on the complexity of the billiard orbits associated with double rotations were obtained.
These estimations are also linear and yield the same bounds as the above.
It is important to get some information about the actual complexity for the most generic class of double rotations of infinite type as well as for wider classes of ITMs.

%\bibliography{biblio.bib}
%\bibliographystyle{alpha}
\printbibliography

\end{document}

%% file: tikz/example.tikz
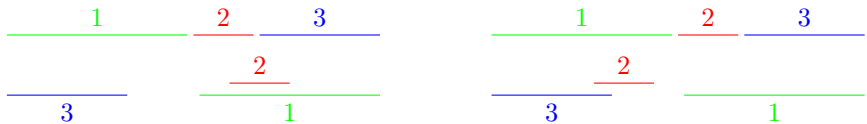
\begin{figure}[bt]
	\def \la {3cm} \def \lb {1cm} \def \lc {2cm} \def \tb {3.7cm} \def \epsi {.1cm} \def \epsip {.2cm}
	\begin{subfigure}{.5\textwidth}
		\begin{tikzpicture}[scale=.8]
			\coordinate (basetop) at (0,0);
			\coordinate (basebot) at (0,-1cm);

			\draw[color=green] (basetop) --++ (\la,0) coordinate (a+) node[pos=.5, above] {$1$};

			\draw (a+) ++ (\epsi,0) coordinate (b-);
			\draw[color=red] (b-) --++ (\lb,0)coordinate(b+) node[pos=.5, above] {$2$};

			\draw (b+) ++ (\epsi,0) coordinate (c-);
			\draw[color=blue] (c-) --++ (\lc,0)coordinate(c+) node[pos=.5, above] {$3$};

			\draw[color=blue] (basebot) --++ (\lc,0) coordinate (c1+) node[pos=.5, below] {$3$};

			\draw (basebot) ++ (\tb,\epsip) coordinate (b1-);
			\draw[color=red] (b1-) --++ (\lb,0)coordinate(b1+) node[pos=.5, above] {$2$};

			\draw (c1+) ++ (\lb + 2*\epsi,0) coordinate (a1-);
			\draw[color=green] (a1-) --++ (\la,0)coordinate(a1+) node[pos=.5, below] {$1$};
		\end{tikzpicture}
	\end{subfigure}
	\begin{subfigure}{.4\textwidth}
		\def \tb {1.7cm}
		\begin{tikzpicture}[scale=.8]
			\coordinate (basetop) at (0,0);
			\coordinate (basebot) at (0,-1cm);

			\draw[color=green]  (basetop) --++ (\la,0) coordinate (a+) node[pos=.5, above] {$1$};

			\draw (a+) ++ (\epsi,0) coordinate (b-);
			\draw[color=red]  (b-) --++ (\lb,0)coordinate(b+) node[pos=.5, above] {$2$};

			\draw (b+) ++ (\epsi,0) coordinate (c-);
			\draw[color=blue]  (c-) --++ (\lc,0)coordinate(c+) node[pos=.5, above] {$3$};

			\draw[color=blue] (basebot) --++ (\lc,0) coordinate (c1+) node[pos=.5, below] {$3$};

			\draw (basebot) ++ (\tb,\epsip) coordinate (b1-);
			\draw[color=red]  (b1-) --++ (\lb,0)coordinate(b1+) node[pos=.5, above] {$2$};

			\draw (c1+) ++ (\lb + 2*\epsi,0) coordinate (a1-);
			\draw[color=green] (a1-) --++ (\la,0)coordinate(a1+) node[pos=.5, below] {$1$};
		\end{tikzpicture}
	\end{subfigure}
	\caption{Examples of 3-ITMs.}
	\label{3itm}
\end{figure}

%% file: tikz/reduction.tikz
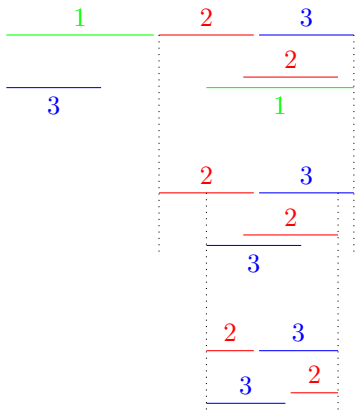
\begin{figure}[bt]
	\def \la {2.8cm} \def \lb {1.8cm} \def \lc {1.8cm} \def \tb {4.5cm} \def \epsi {.1cm} \def \epsip {.2cm}
	\begin{subfigure}{.45\textwidth}
		\begin{tikzpicture}[scale=.7]
			\coordinate (basetop) at (0,0);
			\coordinate (basebot) at (0,-1cm);

			\draw[color=green] (basetop) --++ (\la,0) coordinate (a+) node[pos=.5, above] {$1$};

			\draw (a+) ++ (\epsi,0) coordinate (b-);
			\draw[color=red] (b-) --++ (\lb,0)coordinate(b+) node[pos=.5, above] {$2$};

			\draw (b+) ++ (\epsi,0) coordinate (c-);
			\draw[color=blue] (c-) --++ (\lc,0)coordinate(c+) node[pos=.5, above] {$3$};

			\draw[color=blue] (basebot) --++ (\lc,0) coordinate (c1+) node[pos=.5, below] {$3$};

			\draw (basebot) ++ (\tb,\epsip) coordinate (b1-);
			\draw[color=red] (b1-) --++ (\lb,0)coordinate(b1+) node[pos=.5, above] {$2$};

			\draw (c1+) ++ (\lb + 2*\epsi,0) coordinate (a1-);
			\draw[color=green] (a1-) --++ (\la,0)coordinate(a1+) node[pos=.5, below] {$1$};

			\coordinate (basetop) at (0,-3cm);
			\coordinate (basebot) at (0,-4cm);

			\draw (basetop) ++ (\la,0) coordinate (a+);

			\draw (a+) ++ (\epsi,0) coordinate (b-);
			\draw[color=red] (b-) --++ (\lb,0)coordinate(b+) node[pos=.5, above] {$2$};

			\draw (b+) ++ (\epsi,0) coordinate (c-);
			\draw[color=blue] (c-) --++ (\lc,0)coordinate(c+) node[pos=.5, above] {$3$};

			\draw (basebot) ++ (\lc,0) coordinate (c1+);

			\draw (basebot) ++ (\tb,\epsip) coordinate (b1-);
			\draw[color=red] (b1-) --++ (\lb,0)coordinate(b1+) node[pos=.5, above] {$2$};

			\draw (c1+) ++ (\lb + 2*\epsi,0) coordinate (a1-);
			\draw[color=blue] (a1-) --++ (\lc,0)coordinate(a1+) node[pos=.5, below] {$3$};

			\draw  (b-) ++ (0,-1.2cm) coordinate (b-below);
			\draw  (c+) ++ (0,-1.2cm) coordinate (c+below);
			\draw[dotted] (\la+\epsi,0) -- (b-below);
			\draw[dotted] (\la+\lb+\lc+2*\epsi,0) -- (c+below);

			\draw  (b1+) ++ (0,1cm-\epsip) coordinate (b1+above);
			\draw  (a1-) ++ (0,1cm) coordinate (a1-above);

			\coordinate (basetop) at (0,-6cm);
			\coordinate (basebot) at (0,-7cm);

			\draw (basetop) ++ (\lc+\lb+2*\epsi,0) coordinate (b-);
			\draw (basetop) ++ (\la+\lb+\epsi,0) coordinate (b+);
			\draw[color=red] (b-) -- (b+) node[pos=.5, above] {$2$};

			\draw (b+) ++ (\epsi,0) coordinate (c-);
			\draw (basetop) ++ (\tb+\lb,0) coordinate (c+);
			\draw[color=blue] (c-) -- (c+) node[pos=.5, above] {$3$};

			\draw (basebot) ++ (\lc,0) coordinate (c1+);

			\draw (basebot) ++ (\tb,\epsip) coordinate (b1-);
			\draw (b1-) ++ (\lb,0) coordinate (b1+);

			\draw (c1+) ++ (\lb + 2*\epsi,0) coordinate (a1-);
			\draw (a1-) ++ (\lc,0) coordinate(a1+);

			\draw[color=red] (b1+) --++ ({-(\la-\lc-\epsi)}, 0) coordinate (b1p-) node[pos=.5, above] {$2$};
			\draw (b1p-) ++ ({-\epsi}, -\epsip) coordinate (c1p+);
			\draw[color=blue] (c1p+) --++ ({-\tb-\lb+(\la+\lb+2*\epsi)}, 0)coordinate (c1p-) node[pos=.5, above] {$3$};

			\draw  (b1+) ++ (0,-.3cm-\epsip) coordinate (b1+below);
			\draw  (a1-) ++ (0,-.3cm) coordinate (a1-below);
			\draw[dotted] (b1+above) -- (b1+below);
			\draw[dotted] (a1-above) -- (a1-below);

		\end{tikzpicture}
	\end{subfigure}
	\caption{Reduction of a $3$-ITM to a rotation.}
	\label{fig:red}
\end{figure}

%% file: tikz/rauzy_diagram.tikz
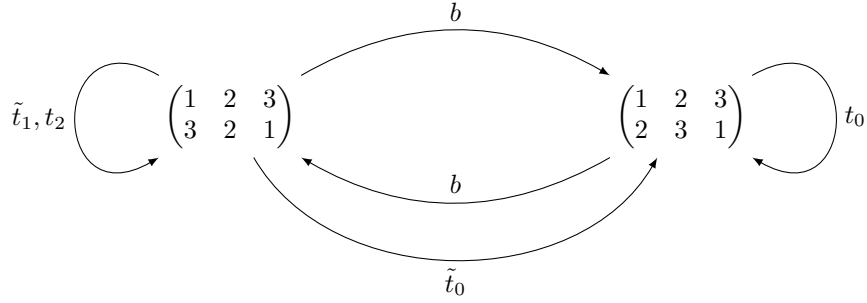
\begin{figure}[bt]
	\begin{tikzpicture}[scale=.7,>=stealth]
		\node (left) [left=2cm]
		{
			$
				\begin{pmatrix}
					1 & 2 & 3 \\
					3 & 2 & 1
				\end{pmatrix}
			$
		};
		\node (right) [right=2cm]{
			$
				\begin{pmatrix}
					1 & 2 & 3 \\
					2 & 3 & 1
				\end{pmatrix}
			$
		};
		% connect the nodes
		\draw[->,>=latex] (left) to[bend left]
		node[above] {$b$}
		(right);
		\draw[->,>=latex] (right) to[bend left]
		node[above] {$b$}
		(left);
		\draw[->,>=latex] (left) to[out=300, in= 240, looseness=1]
		node[below] {$\tilde t_0$}
		(right);
		\draw[->,>=latex] (right) to [out=30, in=330, looseness=4]
		node[right] {$t_0$}
		(right);
		\draw[->,>=latex] (left) to [out=150, in=210, looseness=4]
		node[left] {$\tilde t_1, t_2$}
		(left);
	\end{tikzpicture}
	\caption{The Rauzy Diagram of the \zorich-induction.}
	\label{fig:zdiagram}
\end{figure}

%% file: tikz/first_combinatoric.tikz
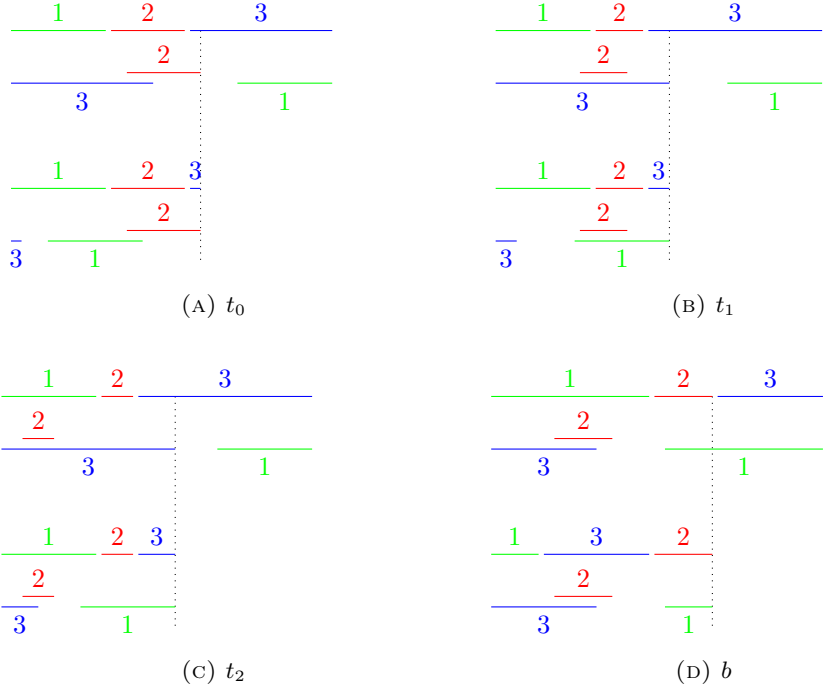
\begin{figure}[t]
	\def \la {1.8cm} \def \lb {1.4cm} \def \lc {2.7cm} \def \tb {2.2cm} \def \epsi {.1cm} \def \epsip {.2cm}
	\begin{subfigure}{.45\textwidth}
		\begin{tikzpicture}[scale=.7]
			\coordinate (basetop) at (0,0);
			\coordinate (basebot) at (0,-1cm);

			\draw[dotted] (\tb+\lb,0) -++ (0,-4.4cm);

			\draw[color=green] (basetop) --++ (\la,0) coordinate (a+) node[pos=.5, above] {$1$};

			\draw (a+) ++ (\epsi,0) coordinate (b-);
			\draw[color=red] (b-) --++ (\lb,0)coordinate(b+) node[pos=.5, above] {$2$};

			\draw (b+) ++ (\epsi,0) coordinate (c-);
			\draw[color=blue] (c-) --++ (\lc,0)coordinate(c+) node[pos=.5, above] {$3$};

			\draw[color=blue] (basebot) --++ (\lc,0) coordinate (c1+) node[pos=.5, below] {$3$};

			\draw (basebot) ++ (\tb,\epsip) coordinate (b1-);
			\draw[color=red] (b1-) --++ (\lb,0)coordinate(b1+) node[pos=.5, above] {$2$};

			\draw (c1+) ++ (\lb + 2*\epsi,0) coordinate (a1-);
			\draw[color=green] (a1-) --++ (\la,0)coordinate(a1+) node[pos=.5, below] {$1$};

			\coordinate (basetop) at (0,-3cm);
			\coordinate (basebot) at (0,-4cm);

			\draw[color=green]  (basetop) --++ (\la,0) coordinate (a+) node[pos=.5, above] {$1$};

			\draw (a+) ++ (\epsi,0) coordinate (b-);
			\draw[color=red]  (b-) --++ (\lb,0)coordinate(b+) node[pos=.5, above] {$2$};

			\draw (b+) ++ (\epsi,0) coordinate (c-);
			\draw[color=blue]  (c-) --++ ({\lc-(\lc-\tb+\la+2*\epsi)},0)coordinate(c+) node[pos=.5, above] {$3$};

			\draw[color=blue] (basebot) --++ ({\lc-(\lc-\tb+\la+2*\epsi)},0) coordinate (c1+) node[pos=.5, below] {$3$};

			\draw (basebot) ++ (\tb,\epsip) coordinate (b1-);
			\draw[color=red]  (b1-) --++ (\lb,0)coordinate(b1+) node[pos=.5, above] {$2$};

			\draw (c1+) ++ (\lc - \tb,0) coordinate (a1-);
			\draw[color=green] (a1-) --++ (\la,0)coordinate(a1+) node[pos=.5, below] {$1$};
		\end{tikzpicture}
		\subcaption{$t_0$}
	\end{subfigure}
	\def \la {1.8cm} \def \lb {.9cm} \def \lc {3.3cm} \def \tb {1.6cm} \def \epsi {.1cm} \def \epsip {.2cm}
	\begin{subfigure}{.45\textwidth}
		\begin{tikzpicture}[scale=.7]
			\coordinate (basetop) at (0,0);
			\coordinate (basebot) at (0,-1cm);

			\draw[dotted] (\lc,0) -++ (0,-4.4cm);

			\draw[color=green] (basetop) --++ (\la,0) coordinate (a+) node[pos=.5, above] {$1$};

			\draw (a+) ++ (\epsi,0) coordinate (b-);
			\draw[color=red] (b-) --++ (\lb,0)coordinate(b+) node[pos=.5, above] {$2$};

			\draw (b+) ++ (\epsi,0) coordinate (c-);
			\draw[color=blue] (c-) --++ (\lc,0)coordinate(c+) node[pos=.5, above] {$3$};

			\draw[color=blue] (basebot) --++ (\lc,0) coordinate (c1+) node[pos=.5, below] {$3$};

			\draw (basebot) ++ (\tb,\epsip) coordinate (b1-);
			\draw[color=red] (b1-) --++ (\lb,0)coordinate(b1+) node[pos=.5, above] {$2$};

			\draw (c1+) ++ (\lb + 2*\epsi,0) coordinate (a1-);
			\draw[color=green] (a1-) --++ (\la,0)coordinate(a1+) node[pos=.5, below] {$1$};

			\coordinate (basetop) at (0,-3cm);
			\coordinate (basebot) at (0,-4cm);

			\draw[color=green]  (basetop) --++ (\la,0) coordinate (a+) node[pos=.5, above] {$1$};

			\draw (a+) ++ (\epsi,0) coordinate (b-);
			\draw[color=red]  (b-) --++ (\lb,0)coordinate(b+) node[pos=.5, above] {$2$};

			\draw (b+) ++ (\epsi,0) coordinate (c-);
			\draw[color=blue]  (c-) --++ ({\lc-(\la+\lb+2*\epsi)},0)coordinate(c+) node[pos=.5, above] {$3$};

			\draw[color=blue] (basebot) --++ ({\lc-(\la+\lb+2*\epsi)},0) coordinate (c1+) node[pos=.5, below] {$3$};

			\draw (basebot) ++ (\tb,\epsip) coordinate (b1-);
			\draw[color=red]  (b1-) --++ (\lb,0)coordinate(b1+) node[pos=.5, above] {$2$};

			\draw (c1+) ++ (\lb+2*\epsi,0) coordinate (a1-);
			\draw[color=green] (a1-) --++ (\la,0)coordinate(a1+) node[pos=.5, below] {$1$};
		\end{tikzpicture}
		\subcaption{$t_1$}
	\end{subfigure}
	\def \la {1.8cm} \def \lb {.6cm} \def \lc {3.3cm} \def \tb {.4cm} \def \epsi {.1cm} \def \epsip {.2cm}
	\begin{subfigure}{.45\textwidth}
		\vspace*{.5cm}
		\begin{tikzpicture}[scale=.7]
			\coordinate (basetop) at (0,0);
			\coordinate (basebot) at (0,-1cm);

			\draw[dotted] (\lc,0) -++ (0,-4.4cm);

			\draw[color=green] (basetop) --++ (\la,0) coordinate (a+) node[pos=.5, above] {$1$};

			\draw (a+) ++ (\epsi,0) coordinate (b-);
			\draw[color=red] (b-) --++ (\lb,0)coordinate(b+) node[pos=.5, above] {$2$};

			\draw (b+) ++ (\epsi,0) coordinate (c-);
			\draw[color=blue] (c-) --++ (\lc,0)coordinate(c+) node[pos=.5, above] {$3$};

			\draw[color=blue] (basebot) --++ (\lc,0) coordinate (c1+) node[pos=.5, below] {$3$};

			\draw (basebot) ++ (\tb,\epsip) coordinate (b1-);
			\draw[color=red] (b1-) --++ (\lb,0)coordinate(b1+) node[pos=.5, above] {$2$};

			\draw (c1+) ++ (\lb + 2*\epsi,0) coordinate (a1-);
			\draw[color=green] (a1-) --++ (\la,0)coordinate(a1+) node[pos=.5, below] {$1$};

			\coordinate (basetop) at (0,-3cm);
			\coordinate (basebot) at (0,-4cm);

			\draw[color=green]  (basetop) --++ (\la,0) coordinate (a+) node[pos=.5, above] {$1$};

			\draw (a+) ++ (\epsi,0) coordinate (b-);
			\draw[color=red]  (b-) --++ (\lb,0)coordinate(b+) node[pos=.5, above] {$2$};

			\draw (b+) ++ (\epsi,0) coordinate (c-);
			\draw[color=blue]  (c-) --++ ({\lc-(\la+\lb+2*\epsi)},0)coordinate(c+) node[pos=.5, above] {$3$};

			\draw[color=blue] (basebot) --++ ({\lc-(\la+\lb+2*\epsi)},0) coordinate (c1+) node[pos=.5, below] {$3$};

			\draw (basebot) ++ (\tb,\epsip) coordinate (b1-);
			\draw[color=red]  (b1-) --++ (\lb,0)coordinate(b1+) node[pos=.5, above] {$2$};

			\draw (c1+) ++ (\lb+2*\epsi,0) coordinate (a1-);
			\draw[color=green] (a1-) --++ (\la,0)coordinate(a1+) node[pos=.5, below] {$1$};
		\end{tikzpicture}
		\subcaption{$t_2$}
	\end{subfigure}
	\def \la {3cm} \def \lb {1.1cm} \def \lc {2cm} \def \tb {1.2cm} \def \epsi {.1cm} \def \epsip {.2cm}
	\begin{subfigure}{.45\textwidth}
		\vspace*{.5cm}
		\begin{tikzpicture}[scale=.7]
			\coordinate (basetop) at (0,0);
			\coordinate (basebot) at (0,-1cm);

			\draw[dotted] (\la+\lb+\epsi,0) -++ (0,-4.4cm);

			\draw[color=green] (basetop) --++ (\la,0) coordinate (a+) node[pos=.5, above] {$1$};

			\draw (a+) ++ (\epsi,0) coordinate (b-);
			\draw[color=red] (b-) --++ (\lb,0)coordinate(b+) node[pos=.5, above] {$2$};

			\draw (b+) ++ (\epsi,0) coordinate (c-);
			\draw[color=blue] (c-) --++ (\lc,0)coordinate(c+) node[pos=.5, above] {$3$};

			\draw[color=blue] (basebot) --++ (\lc,0) coordinate (c1+) node[pos=.5, below] {$3$};

			\draw (basebot) ++ (\tb,\epsip) coordinate (b1-);
			\draw[color=red] (b1-) --++ (\lb,0)coordinate(b1+) node[pos=.5, above] {$2$};

			\draw (c1+) ++ (\lb + 2*\epsi,0) coordinate (a1-);
			\draw[color=green] (a1-) --++ (\la,0)coordinate(a1+) node[pos=.5, below] {$1$};

			\coordinate (basetop) at (0,-3cm);
			\coordinate (basebot) at (0,-4cm);

			\draw[color=green] (basetop) --++ (\la-\lc-\epsi,0) coordinate (a+) node[pos=.5, above] {$1$};

			\draw (a+) ++ (\epsi,0) coordinate (c-);
			\draw[color=blue] (c-) --++ (\lc,0)coordinate(c+) node[pos=.5, above] {$3$};

			\draw (c+) ++ (\epsi,0) coordinate (b-);
			\draw[color=red] (b-) --++ (\lb,0)coordinate(b+) node[pos=.5, above] {$2$};

			\draw[color=blue] (basebot) --++ (\lc,0) coordinate (c1+) node[pos=.5, below] {$3$};

			\draw (basebot) ++ (\tb,\epsip) coordinate (b1-);
			\draw[color=red] (b1-) --++ (\lb,0)coordinate(b1+) node[pos=.5, above] {$2$};

			\draw (c1+) ++ (\lb + 2*\epsi,0) coordinate (a1-);
			\draw[color=green] (a1-) --++ (\la-\lc-\epsi,0)coordinate(a1+) node[pos=.5, below] {$1$};
		\end{tikzpicture}
		\subcaption{$b$}
	\end{subfigure}
	\caption{The \zorich-induction for the first combinatoric.}
	\label{fig:z-induction}
\end{figure}

%% file: tikz/second_combinatoric.tikz
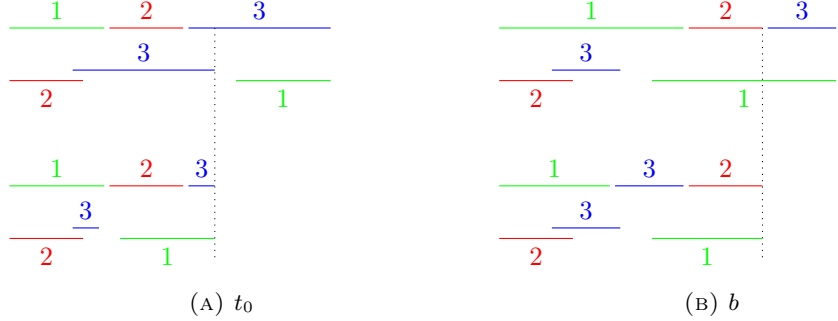
\begin{figure}[bt]
	\def \la {1.8cm} \def \lb {1.4cm} \def \lc {2.7cm} \def \tc {1.2cm} \def \epsi {.1cm} \def \epsip {.2cm}
	\begin{subfigure}{.45\textwidth}
		\begin{tikzpicture}[scale=.7]
			\coordinate (basetop) at (0,0);
			\coordinate (basebot) at (0,-1cm);

			\draw[dotted] (\tc+\lc,0) -++ (0,-4.4cm);

			\draw[color=green] (basetop) --++ (\la,0) coordinate (a+) node[pos=.5, above] {$1$};

			\draw (a+) ++ (\epsi,0) coordinate (b-);
			\draw[color=red] (b-) --++ (\lb,0)coordinate(b+) node[pos=.5, above] {$2$};

			\draw (b+) ++ (\epsi,0) coordinate (c-);
			\draw[color=blue] (c-) --++ (\lc,0)coordinate(c+) node[pos=.5, above] {$3$};

			\draw[color=red] (basebot) --++ (\lb,0) coordinate (b1+) node[pos=.5, below] {$2$};

			\draw (basebot) ++ (\tc,\epsip) coordinate (c1-);
			\draw[color=blue] (c1-) --++ (\lc,0)coordinate(c1+) node[pos=.5, above] {$3$};

			\draw (b1+) ++ (\lc + 2*\epsi,0) coordinate (a1-);
			\draw[color=green] (a1-) --++ (\la,0)coordinate(a1+) node[pos=.5, below] {$1$};

			\coordinate (basetop) at (0,-3cm);
			\coordinate (basebot) at (0,-4cm);

			\draw[color=green] (basetop) --++ (\la,0) coordinate (a+) node[pos=.5, above] {$1$};

			\draw (a+) ++ (\epsi,0) coordinate (b-);
			\draw[color=red] (b-) --++ (\lb,0)coordinate(b+) node[pos=.5, above] {$2$};

			\draw (b+) ++ (\epsi,0) coordinate (c-);
			\draw[color=blue] (c-) --++ ({\lc-\la-(\lb-\tc)-2*\epsi},0)coordinate(c+) node[pos=.5, above] {$3$};

			\draw[color=red] (basebot) --++ (\lb,0) coordinate (b1+) node[pos=.5, below] {$2$};

			\draw (basebot) ++ (\tc,\epsip) coordinate (c1-);
			\draw[color=blue] (c1-) --++ ({\lc-\la-(\lb-\tc)-2*\epsi},0)coordinate(c1+) node[pos=.5, above] {$3$};

			\draw (b1+) ++ ({\lc-\la-(\lb-\tc)},0) coordinate (a1-);
			\draw[color=green] (a1-) --++ (\la,0)coordinate(a1+) node[pos=.5, below] {$1$};
		\end{tikzpicture}
		\subcaption{$t_0$}
	\end{subfigure}
	\def \la {3.5cm} \def \lb {1.4cm} \def \lc {1.3cm} \def \tc {1.0cm} \def \epsi {.1cm} \def \epsip {.2cm}
	\begin{subfigure}{.45\textwidth}
		\begin{tikzpicture}[scale=.7]
			\coordinate (basetop) at (0,0);
			\coordinate (basebot) at (0,-1cm);

			\draw[dotted] (\la+\lb+\epsi,0) -++ (0,-4.4cm);

			\draw[color=green] (basetop) --++ (\la,0) coordinate (a+) node[pos=.5, above] {$1$};

			\draw (a+) ++ (\epsi,0) coordinate (b-);
			\draw[color=red] (b-) --++ (\lb,0)coordinate(b+) node[pos=.5, above] {$2$};

			\draw (b+) ++ (\epsi,0) coordinate (c-);
			\draw[color=blue] (c-) --++ (\lc,0)coordinate(c+) node[pos=.5, above] {$3$};

			\draw[color=red] (basebot) --++ (\lb,0) coordinate (b1+) node[pos=.5, below] {$2$};

			\draw (basebot) ++ (\tc,\epsip) coordinate (c1-);
			\draw[color=blue] (c1-) --++ (\lc,0)coordinate(c1+) node[pos=.5, above] {$3$};

			\draw (b1+) ++ (\lc + 2*\epsi,0) coordinate (a1-);
			\draw[color=green] (a1-) --++ (\la,0)coordinate(a1+) node[pos=.5, below] {$1$};

			\coordinate (basetop) at (0,-3cm);
			\coordinate (basebot) at (0,-4cm);

			\draw[color=green] (basetop) --++ (\la-\lc-\epsi,0) coordinate (a+) node[pos=.5, above] {$1$};

			\draw (a+) ++ (\epsi,0) coordinate (c-);
			\draw[color=blue] (c-) --++ (\lc,0)coordinate(c+) node[pos=.5, above] {$3$};

			\draw (c+) ++ (\epsi,0) coordinate (b-);
			\draw[color=red] (b-) --++ (\lb,0)coordinate(b+) node[pos=.5, above] {$2$};

			\draw[color=red] (basebot) --++ (\lb,0) coordinate (b1+) node[pos=.5, below] {$2$};

			\draw (basebot) ++ (\tc,\epsip) coordinate (c1-);
			\draw[color=blue] (c1-) --++ (\lc,0)coordinate(c1+) node[pos=.5, above] {$3$};

			\draw (b1+) ++ (\lc + 2*\epsi,0) coordinate (a1-);
			\draw[color=green] (a1-) --++ (\la-\lc-\epsi,0)coordinate(a1+) node[pos=.5, below] {$1$};
		\end{tikzpicture}
		\subcaption{$b$}
	\end{subfigure}
	\caption{The \zorich-induction for the second combinatoric.}
	\label{fig:z-induction_rot}
\end{figure}

%% file: tikz/split.tikz
\usetikzlibrary{patterns,snakes}
\begin{figure}[h]
	\def \la {1.8cm} \def \lb {1.4cm} \def \lc {2.7cm} \def \tb {2.2cm} \def \epsi {.1cm} \def \epsip {.2cm}
	\begin{subfigure}{.45\textwidth}
		\begin{tikzpicture}[scale=.8]
			\coordinate (basetop) at (0,0);
			\coordinate (basebot) at (0,-1cm);

			\draw[color=green] (basetop) --++ (\la,0) coordinate (a+) node[pos=.5, above] {$1$};

			\draw (a+) ++ (\epsi,0) coordinate (b-);
			\draw[color=red] (b-) --++ (\lb,0)coordinate(b+) node[pos=.5, above] {$2$};

			\draw (b+) ++ (\epsi,0) coordinate (c-);
			\draw[color=blue] (c-) --++ (\lc,0)coordinate(c+) node[pos=.5, above] {$3$};

			\draw[color=blue] (basebot) --++ (\lc,0) coordinate (c1+) node[pos=.5, below] {$3$};

			\draw (basebot) ++ (\tb,\epsip) coordinate (b1-);
			\draw[color=red] (b1-) --++ (\lb,0)coordinate(b1+) node[pos=.5, above] {$2$};

			\draw (c1+) ++ (\lb + 2*\epsi,0) coordinate (a1-);
			\draw[color=green] (a1-) --++ (\la,0)coordinate(a1+) node[pos=.5, below] {$1$};

			\coordinate (basetop) at (0,-3cm);
			\coordinate (basebot) at (0,-4cm);

			\draw[color=green]  (basetop) --++ (\la,0) coordinate (a+) node[pos=.5, above] {$A$};

			\draw (a+) ++ (\epsi,0) coordinate (b-);
			\draw[color=purple] (b-) --++ (\lc-\tb,0)coordinate(b+) node[pos=.5, above] {$D$};
			\draw (b+) ++ (\epsi,0) coordinate (b+);
			\draw[color=red] (b+) --++ (\lb+\tb-\lc-\epsi,0)coordinate(b++) node[pos=.5, above] {$B$};

			\draw (b++) ++ (\epsi,0) coordinate (c-);
			\draw[color=blue]  (c-) --++ (\tb-\epsi,0)coordinate(c+) node[pos=.5, above] {$C$};
			\draw (c+) ++ (\epsi,0) coordinate (c+);
			\draw[color=purple]  (c+) --++ (\lc-\tb,0)coordinate(c++) node[pos=.5, above] {$D$};

			\draw[color=blue] (basebot) --++ (\tb-\epsi,0) coordinate (c1+) node[pos=.5, below] {$C$};
			\draw (c1+) ++ (\epsi,0) coordinate (c1+);
			\draw[color=purple] (c1+) --++ (\lc-\tb,0) coordinate (c1++) node[pos=.5, below] {$D$};

			\draw (basebot) ++ (\tb,\epsip) coordinate (b1-);
			\draw[color=purple] (b1-) --++ (\lc-\tb,0)coordinate(b1+) node[pos=.5, above] {$D$};
			\draw (b1+) ++ (\epsi,0) coordinate (b1+);
			\draw[color=red] (b1+) --++ (\lb+\tb-\lc-\epsi,0)coordinate(b1++) node[pos=.5, above] {$B$};

			\draw (c1++) ++ (\lb + 2*\epsi,0) coordinate (a1-);
			\draw[color=green] (a1-) --++ (\la,0)coordinate(a1+) node[pos=.5, below] {$A$};
			\draw[thick,decoration={brace,raise=0.5cm},decorate] (a1-) --++ (-\lc+\tb-2*\epsi,0)coordinate(a1+) node [pos=0.5,anchor=north,yshift=-0.55cm] {$D$};
		\end{tikzpicture}
	\end{subfigure}
	\def \la {1.8cm} \def \lb {.9cm} \def \lc {3.3cm} \def \tb {1.6cm} \def \epsi {.1cm} \def \epsip {.2cm}
	\begin{subfigure}{.45\textwidth}
		\begin{tikzpicture}[scale=.8]
			\coordinate (basetop) at (0,0);
			\coordinate (basebot) at (0,-1cm);

			\draw[color=green] (basetop) --++ (\la,0) coordinate (a+) node[pos=.5, above] {$1$};

			\draw (a+) ++ (\epsi,0) coordinate (b-);
			\draw[color=red] (b-) --++ (\lb,0)coordinate(b+) node[pos=.5, above] {$2$};

			\draw (b+) ++ (\epsi,0) coordinate (c-);
			\draw[color=blue] (c-) --++ (\lc,0)coordinate(c+) node[pos=.5, above] {$3$};

			\draw[color=blue] (basebot) --++ (\lc,0) coordinate (c1+) node[pos=.5, below] {$3$};

			\draw (basebot) ++ (\tb,\epsip) coordinate (b1-);
			\draw[color=red] (b1-) --++ (\lb,0)coordinate(b1+) node[pos=.5, above] {$2$};

			\draw (c1+) ++ (\lb + 2*\epsi,0) coordinate (a1-);
			\draw[color=green] (a1-) --++ (\tb,0)coordinate(a1+) node[pos=.5, below] {$1$};

			\coordinate (basetop) at (0,-3cm);
			\coordinate (basebot) at (0,-4cm);

			\draw[color=green] (basetop) --++ (\la,0) coordinate (a+) node[pos=.5, above] {$A$};

			\draw (a+) ++ (\epsi,0) coordinate (b-);
			\draw[color=red] (b-) --++ (\lb,0)coordinate(b+) node[pos=.5, above] {$B$};

			\draw (b+) ++ (\epsi,0) coordinate (c-);
			\draw[color=blue] (c-) --++ (\tb-\epsi,0)coordinate(c+) node[pos=.5, above] {$C$};
			\draw (c+) ++ (\epsi,0) coordinate (c+);
			\draw[color=red] (c+) --++ (\lb,0)coordinate(c+) node[pos=.5, above] {$B$};
			\draw (c+) ++ (\epsi,0) coordinate (c+);
			\draw[color=blue] (c+) --++ (\lc-\tb-\lb-\epsi,0)coordinate(c+) node[pos=.5, above] {$D$};

			\draw[color=blue] (basebot) --++ (\tb-\epsi,0) coordinate (c+) node[pos=.5, below] {$C$};
			\draw (c+) ++ (\epsi,0) coordinate (c+);
			\draw[color=red] (c+) --++ (\lb,0)coordinate(c+) node[pos=.5, below] {$B$};
			\draw (c+) ++ (\epsi,0) coordinate (c+);
			\draw[color=blue] (c+) --++ (\lc-\tb-\lb-\epsi,0)coordinate(c+) node[pos=.5, below] {$D$};

			\draw (basebot) ++ (\tb,\epsip) coordinate (b1-);
			\draw[color=red] (b1-) --++ (\lb,0)coordinate(b1+) node[pos=.5, above] {$B$};

			\draw (c+) ++ (\lb + 2*\epsi,0) coordinate (a1-);
			\draw[color=green] (a1-) --++ (\la,0)coordinate(a1+) node[pos=.5, below] {$A$};

			\draw[thick,decoration={brace,raise=0.5cm},decorate] (a1-) --++ (-\lb-2*\epsi,0)coordinate(a1+) node [pos=0.5,anchor=north,yshift=-0.55cm] {$B$};
		\end{tikzpicture}
	\end{subfigure}
	\caption{The splitting of 3-ITMs.}
	\label{fig:split}
\end{figure}
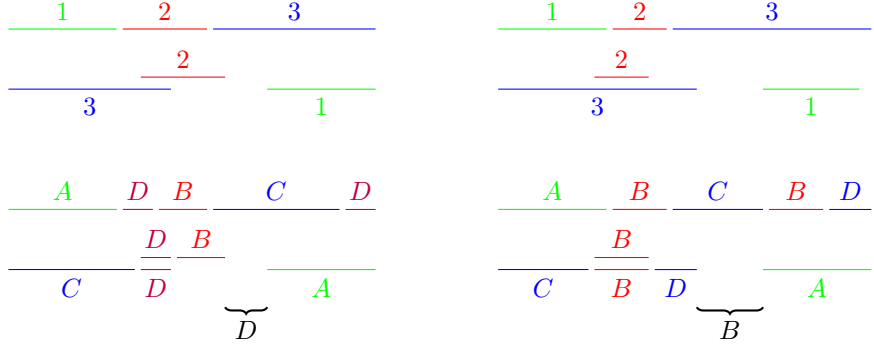

%% file: tikz/rauzy.tikz
\begin{figure}[h]
	\def \la {1.8cm} \def \lb {1.4cm} \def \lc {2.7cm} \def \tb {2.2cm} \def \epsi {.1cm} \def \epsip {.2cm}
	\def \lA {\la-\lc+\tb-\epsi}
	\begin{subfigure}{.45\textwidth}
		\begin{tikzpicture}[scale=.8]
			\coordinate (basetop) at (0,0cm);
			\coordinate (basebot) at (0,-1cm);

			\draw[dotted] (\la+\lb+\tb+\epsi,0) -++ (0,-4.4cm);

			\draw[color=green]  (basetop) --++ (\la,0) coordinate (a+) node[pos=.5, above] {$A$};

			\draw (a+) ++ (\epsi,0) coordinate (b-);
			\draw[color=purple] (b-) --++ (\lc-\tb,0)coordinate(b+) node[pos=.5, above] {$B$};
			\draw (b+) ++ (\epsi,0) coordinate (b+);
			\draw[color=red] (b+) --++ (\lb+\tb-\lc-\epsi,0)coordinate(b++) node[pos=.5, above] {$D$};

			\draw (b++) ++ (\epsi,0) coordinate (c-);
			\draw[color=blue]  (c-) --++ (\tb-\epsi,0)coordinate(c+) node[pos=.5, above] {$C$};
			\draw (c+) ++ (\epsi,0) coordinate (c+);
			\draw[color=purple]  (c+) --++ (\lc-\tb,0)coordinate(c++) node[pos=.5, above] {$B$};

			\draw[color=blue] (basebot) --++ (\tb-\epsi,0) coordinate (c1+) node[pos=.5, below] {$C$};
			\draw (c1+) ++ (\epsi,0) coordinate (c1+);
			\draw[color=purple] (c1+) --++ (\lc-\tb,0) coordinate (c1++) node[pos=.5, below] {$B$};

			\draw (c1++) ++ (\epsi,0) coordinate (b1+);
			\draw[color=red] (b1+) --++ (\lb+\tb-\lc-\epsi,0)coordinate(b1++) node[pos=.5, below] {$D$};

			\draw (c1++) ++ (\lb + 2*\epsi,0) coordinate (a1-);
			\draw[color=green] (a1-) --++ (\la,0)coordinate(a1+) node[pos=.5, below] {$A$};

			\coordinate (basetop) at (0,-3cm);
			\coordinate (basebot) at (0,-4cm);

			\draw[color=green]  (basetop) --++ (\lA,0) coordinate (a+) node[pos=.5, above] {$A$};

			\draw (a+) ++ (\epsi,0) coordinate (b-);
			\draw[color=purple] (b-) --++ (\lc-\tb,0)coordinate(b+) node[pos=.5, above] {$B$};
			\draw (b+) ++ (\epsi,0) coordinate (b+);
			\draw[color=purple] (b+) --++ (\lc-\tb,0) coordinate(b+) node[pos=.5, above] {$B$};
			\draw (b+) ++ (\epsi,0) coordinate (b+);
			\draw[color=red] (b+) --++ (\lb+\tb-\lc-\epsi,0)coordinate(b++) node[pos=.5, above] {$D$};

			\draw (b++) ++ (\epsi,0) coordinate (c-);
			\draw[color=blue]  (c-) --++ (\tb-\epsi,0)coordinate(c+) node[pos=.5, above] {$C$};

			\draw[color=blue] (basebot) --++ (\tb-\epsi,0) coordinate (c1+) node[pos=.5, below] {$C$};
			\draw (c1+) ++ (\epsi,0) coordinate (c1+);
			\draw[color=purple] (c1+) --++ (\lc-\tb,0) coordinate (c1++) node[pos=.5, below] {$B$};

			\draw (c1++) ++ (\epsi,0) coordinate (b1+);
			\draw[color=red] (b1+) --++ (\lb+\tb-\lc-\epsi,0)coordinate(b1++) node[pos=.5, below] {$D$};

			\draw (c1++) ++ (\lb + 2*\epsi,0) coordinate (a1-);
			\draw[color=green] (a1-) --++ (\lA,0)coordinate(a1+) node[pos=.5, below] {$A$};
		\end{tikzpicture}
		\subcaption{The winner occurs once.}
		\label{fig:RauzyOnce}
	\end{subfigure}
	\def \la {.5cm} \def \lb {2.4cm} \def \lc {2.2cm} \def \tb {.9cm} \def \epsi {.1cm} \def \epsip {.2cm}
	\def \lA {\la-\lc+\tb-\epsi}
	\begin{subfigure}{.45\textwidth}
		\begin{tikzpicture}[scale=.8]
			\coordinate (basetop) at (0,0cm);
			\coordinate (basebot) at (0,-1cm);

			\draw[dotted] (\la+\lb+\tb+\epsi,0) -++ (0,-4.4cm);

			\draw[color=green]  (basetop) --++ (\la,0) coordinate (a+) node[pos=.5, above] {$A$};

			\draw (a+) ++ (\epsi,0) coordinate (b-);
			\draw[color=purple] (b-) --++ (\lc-\tb,0)coordinate(b+) node[pos=.5, above] {$B$};
			\draw (b+) ++ (\epsi,0) coordinate (b+);
			\draw[color=red] (b+) --++ (\lb+\tb-\lc-\epsi,0)coordinate(b++) node[pos=.5, above] {$D$};

			\draw (b++) ++ (\epsi,0) coordinate (c-);
			\draw[color=blue]  (c-) --++ (\tb-\epsi,0)coordinate(c+) node[pos=.5, above] {$C$};
			\draw (c+) ++ (\epsi,0) coordinate (c+);
			\draw[color=purple]  (c+) --++ (\lc-\tb,0)coordinate(c++) node[pos=.5, above] {$B$};

			\draw[color=blue] (basebot) --++ (\tb-\epsi,0) coordinate (c1+) node[pos=.5, below] {$C$};
			\draw (c1+) ++ (\epsi,0) coordinate (c1+);
			\draw[color=purple] (c1+) --++ (\lc-\tb,0) coordinate (c1++) node[pos=.5, below] {$B$};

			\draw (c1++) ++ (\epsi,0) coordinate (b1+);
			\draw[color=red] (b1+) --++ (\lb+\tb-\lc-\epsi,0)coordinate(b1++) node[pos=.5, below] {$D$};

			\draw (c1++) ++ (\lb + 2*\epsi,0) coordinate (a1-);
			\draw[color=green] (a1-) --++ (\la,0)coordinate(a1+) node[pos=.5, below] {$A$};

			\coordinate (basetop) at (0,-3cm);
			\coordinate (basebot) at (0,-4cm);

			\draw[color=green]  (basetop) --++ (\la,0) coordinate (a+) node[pos=.5, above] {$A$};

			\draw (a+) ++ (\epsi,0) coordinate (b-);
			\draw[color=purple] (b-) --++ (\lc-\tb-\la-\epsi,0)coordinate(b+) node[pos=.5, above] {$B$};
			\draw (b+) ++ (\epsi,0) coordinate (b+);
			\draw[color=green] (b+) --++ (\la,0)coordinate(b+) node[pos=.5, above] {$A$};
			\draw (b+) ++ (\epsi,0) coordinate (b+);
			\draw[color=red] (b+) --++ (\lb+\tb-\lc-\epsi,0)coordinate(b++) node[pos=.5, above] {$D$};

			\draw (b++) ++ (\epsi,0) coordinate (c-);
			\draw[color=blue]  (c-) --++ (\tb-\epsi,0)coordinate(c+) node[pos=.5, above] {$C$};

			\draw[color=blue] (basebot) --++ (\tb-\epsi,0) coordinate (c1+) node[pos=.5, below] {$C$};
			\draw (c1+) ++ (\epsi,0) coordinate (c1+);
			\draw[color=purple] (c1+) --++ (\lc-\tb-\la-\epsi,0) coordinate (c1+) node[pos=.5, below] {$B$};
			\draw (c1+) ++ (\epsi,0) coordinate (c1+);
			\draw[color=green] (c1+) --++ (\la,0) coordinate (c1++) node[pos=.5, below] {$A$};

			\draw (c1++) ++ (\epsi,0) coordinate (b1+);
			\draw[color=red] (b1+) --++ (\lb+\tb-\lc-\epsi,0)coordinate(b1++) node[pos=.5, below] {$D$};
		\end{tikzpicture}
		\subcaption{The winner occurs twice.}
		\label{fig:RauzyTwice}
	\end{subfigure}
	\caption{The Right \rauzy-induction.}
\end{figure}

%% file: biblio.bib
@article {BoshernitzanKornfeld95, AUTHOR = {Boshernitzan, M. and Kornfeld, I.}, TITLE = {Interval translation mappings}, JOURNAL = {Ergodic Theory Dynam. Systems}, FJOURNAL = {Ergodic Theory and Dynamical Systems}, VOLUME = {15}, YEAR = {1995}, NUMBER = {5}, PAGES = {821--832}, ISSN = {0143-3857}, MRCLASS = {58F03 (58F11)}, MRNUMBER = {1356616}, MRREVIEWER = {Jos\'{e} Miguel Moreno}, DOI = {10.1017/S0143385700009652}, URL = {https://doi.org/10.1017/S0143385700009652}, }

@article {Bruin07, AUTHOR = {Bruin, H.}, TITLE = {Renormalization in a class of interval translation maps of {$d$} branches}, JOURNAL = {Dyn. Syst.}, FJOURNAL = {Dynamical Systems. An International Journal}, VOLUME = {22}, YEAR = {2007}, NUMBER = {1}, PAGES = {11--24}, ISSN = {1468-9367}, MRCLASS = {37E20 (37E05 37E10)}, MRNUMBER = {2308208}, MRREVIEWER = {J\'{e}r\^{o}me Buzzi}, DOI = {10.1080/14689360601028084}, URL = {https://doi.org/10.1080/14689360601028084},  shorthand = {B07},}

@article {BruinClack12, AUTHOR = {Bruin, H. and Clack, G.}, TITLE = {Inducing and unique ergodicity of double rotations}, JOURNAL = {Discrete Contin. Dyn. Syst.}, FJOURNAL = {Discrete and Continuous Dynamical Systems. Series A}, VOLUME = {32}, YEAR = {2012}, NUMBER = {12}, PAGES = {4133--4147}, ISSN = {1078-0947}, MRCLASS = {37B10 (37C70 37D25 37E05 37E10)}, MRNUMBER = {2966738}, MRREVIEWER = {David Ralston}, DOI = {10.3934/dcds.2012.32.4133}, URL = {https://doi.org/10.3934/dcds.2012.32.4133}, }

@article {BruinTroubetzkoy03, AUTHOR = {Bruin, H. and Troubetzkoy, S.}, TITLE = {The {G}auss map on a class of interval translation mappings}, JOURNAL = {Israel J. Math.}, FJOURNAL = {Israel Journal of Mathematics}, VOLUME = {137}, YEAR = {2003}, PAGES = {125--148}, ISSN = {0021-2172}, MRCLASS = {37A25 (37B99 37E10)}, MRNUMBER = {2013352}, MRREVIEWER = {Franco Vivaldi}, DOI = {10.1007/BF02785958}, URL = {https://doi.org/10.1007/BF02785958}, }

@article {BuzziHubert04, AUTHOR = {Buzzi, J. and Hubert, P.}, TITLE = {Piecewise monotone maps without periodic points: rigidity, measures and complexity}, JOURNAL = {Ergodic Theory Dynam. Systems}, FJOURNAL = {Ergodic Theory and Dynamical Systems}, VOLUME = {24}, YEAR = {2004}, NUMBER = {2}, PAGES = {383--405}, ISSN = {0143-3857}, MRCLASS = {37E05 (37A05 37B40 37C40)}, MRNUMBER = {2054049}, MRREVIEWER = {Peter Raith}, DOI = {10.1017/S0143385703000488}, URL = {https://doi.org/10.1017/S0143385703000488}, }

@incollection {CassaigneNicolas10, AUTHOR = {Cassaigne, J. and Nicolas, F.}, TITLE = {Factor complexity}, BOOKTITLE = {Combinatorics, automata and number theory}, SERIES = {Encyclopedia Math. Appl.}, VOLUME = {135}, PAGES = {163--247}, PUBLISHER = {Cambridge Univ. Press, Cambridge}, YEAR = {2010}, MRCLASS = {68R15 (37B15 68-01)}, MRNUMBER = {2759107}, MRREVIEWER = {Patrice S\'{e}\'{e}bold}, }

@article{DynnikovHubertSkripchenko20, AUTHOR = {Dynnikov, I. and Hubert, P. and Skripchenko, A.}, TITLE = {Dynamical systems around the Rauzy gasket and their ergodic properties}, eprint = {2011.15043v1}, eprinttype = {arxiv}, eprintclass = {math.DS}, YEAR = {2020}, Url = {http://arxiv.org/abs/2011.15043v1},}

@article{Fougeron20, AUTHOR = {Fougeron, C.}, TITLE = {Dynamical properties of simplicial systems and continued fraction algorithms}, eprint = {2001.01367}, eprinttype = {arxiv}, eprintclass = {math.DS}, YEAR = {2020}, Url = {http://arxiv.org/abs/2001.01367v1},  shorthand = {F20},}

@article {GaboriauLevittPaulin94, AUTHOR = {Gaboriau, D. and Levitt, G. and Paulin, F.}, TITLE = {Pseudogroups of isometries of {${\bf R}$} and {R}ips' theorem on free actions on {${\bf R}$}-trees}, JOURNAL = {Israel J. Math.}, FJOURNAL = {Israel Journal of Mathematics}, VOLUME = {87}, YEAR = {1994}, NUMBER = {1-3}, PAGES = {403--428}, ISSN = {0021-2172}, MRCLASS = {20E08 (57M07)}, MRNUMBER = {1286836}, MRREVIEWER = {Martin A. Roller}, DOI = {10.1007/BF02773004}, URL = {https://doi.org/10.1007/BF02773004}, }

@book{Pytheas,
 author 	= {Pytheas Fogg, N.},
 title		= {Substitutions in Dynamics, Arithmetics and Combinatorics},
 series		= {Lecture Notes in Mathematics},
 number	= {1794},
 publisher	= {Springer},
 location 	= {Berlin},
 year		= {2002},
 shorthand = {PF02},
}

@incollection {SchmelingTroubetzkoy00, AUTHOR = {Schmeling, J. and Troubetzkoy, S.}, TITLE = {Interval translation mappings}, BOOKTITLE = {Dynamical systems ({L}uminy-{M}arseille, 1998)}, PAGES = {291--302}, PUBLISHER = {World Sci. Publ., River Edge, NJ}, YEAR = {2000}, MRCLASS = {37E05 (37B05 54H20)}, MRNUMBER = {1796167}, MRREVIEWER = {Peter Raith}, }

@article {SkripchenkoTroubetzkoy15, AUTHOR = {Skripchenko, A. and Troubetzkoy, S.}, TITLE = {Polygonal billiards with one sided scattering}, JOURNAL = {Ann. Inst. Fourier (Grenoble)}, FJOURNAL = {Universit\'{e} de Grenoble. Annales de l'Institut Fourier}, VOLUME = {65}, YEAR = {2015}, NUMBER = {5}, PAGES = {1881--1896}, ISSN = {0373-0956}, MRCLASS = {37D50 (37B10 37B40 37C35)}, MRNUMBER = {3449199}, MRREVIEWER = {Andr\'{a}s Kr\'{a}mli}, URL = {http://aif.cedram.org/item?id=AIF_2015__65_5_1881_0}, }

@article {SuzukiItoAihara05, AUTHOR = {Suzuki, H. and Ito, S. and Aihara, K.}, TITLE = {Double rotations}, JOURNAL = {Discrete Contin. Dyn. Syst.}, FJOURNAL = {Discrete and Continuous Dynamical Systems. Series A}, VOLUME = {13}, YEAR = {2005}, NUMBER = {2}, PAGES = {515--532}, ISSN = {1078-0947}, MRCLASS = {37E05 (37B05 37E10 37E45)}, MRNUMBER = {2152403}, MRREVIEWER = {Henk Bruin}, DOI = {10.3934/dcds.2005.13.515}, URL = {https://doi.org/10.3934/dcds.2005.13.515}, }

@article {Troubetzkoy21, AUTHOR = {Troubetzkoy, S.}, TITLE = {Absence of mixing for interval translation mappings and some generalizations}, eprint = {2102.05904}, eprinttype = {arxiv}, eprintclass = {math.DS}, YEAR = {2021}, Url = {http://arxiv.org/abs/2102.05904}, note = {To appear in Trans. Moscow Math. Soc.}, shorthand = {T21},}

@article {Veech82, AUTHOR = {Veech, W. A.}, TITLE = {Gauss measures for transformations on the space of interval exchange maps}, JOURNAL = {Ann. of Math. (2)}, FJOURNAL = {Annals of Mathematics. Second Series}, VOLUME = {115}, YEAR = {1982}, NUMBER = {1}, PAGES = {201--242}, ISSN = {0003-486X}, MRCLASS = {28D05 (58F11)}, MRNUMBER = {644019}, MRREVIEWER = {Michael Keane}, DOI = {10.2307/1971391}, URL = {https://doi.org/10.2307/1971391},  shorthand = {V82},}

@article {Volk14, AUTHOR = {Volk, D.}, TITLE = {Almost every interval translation map of three intervals is finite type}, JOURNAL = {Discrete Contin. Dyn. Syst.}, FJOURNAL = {Discrete and Continuous Dynamical Systems. Series A}, VOLUME = {34}, YEAR = {2014}, NUMBER = {5}, PAGES = {2307--2314}, ISSN = {1078-0947}, MRCLASS = {37C70 (37C05 37C20 37D20 37D45)}, MRNUMBER = {3124735}, DOI = {10.3934/dcds.2014.34.2307}, URL = {https://doi.org/10.3934/dcds.2014.34.2307},  shorthand = {V14},}

@incollection {Yoccoz10, AUTHOR = {Yoccoz, J.-C.}, TITLE = {Interval exchange maps and translation surfaces}, BOOKTITLE = {Homogeneous flows, moduli spaces and arithmetic}, SERIES = {Clay Math. Proc.}, VOLUME = {10}, PAGES = {1--69}, PUBLISHER = {Amer. Math. Soc., Providence, RI}, YEAR = {2010}, MRCLASS = {37-02 (28D05 37A25 37C40 37D25 37D40 37D50)}, MRNUMBER = {2648692}, MRREVIEWER = {Thomas A. Schmidt},  shorthand = {Y10},}
